\documentclass[10pt,reqno]{amsart}

\usepackage{amssymb,latexsym,mathrsfs,pdfsync}
\usepackage{url}
\usepackage{fullpage}
\usepackage{enumerate}[1.)]
\numberwithin{equation}{section}
\usepackage{hyperref}

\renewcommand{\eqref}[1]{(\ref{#1})} 
\theoremstyle{plain} \newtheorem{theorem}{Theorem}[section]

\newtheorem{corollary}[theorem]{Corollary}

\newtheorem{lemma}[theorem]{Lemma}

\newtheorem{proposition}[theorem]{Proposition}

\newtheorem{definition}[theorem]{Definition}

\theoremstyle{remark}
\newtheorem{remark}[theorem]{Remark}

\newtheorem*{claim}{Claim}

\newcommand{\mscL}{\mathscr{L}}
\newcommand{\SL}{\mathrm{SL}}

\newcommand{\PGL}{\mathrm{PGL}}
\newcommand{\Gg}{\mathbb{G}}
\newcommand{\Fp}{\mathbb{F}_p}
\newcommand{\Fpt}{\mathbb{F}^\times_p}
\newcommand{\Cc}{\mathbb{C}}
\newcommand{\Rr}{\mathbb{R}}
\newcommand{\Zz}{\mathbb{Z}}
\newcommand{\Ff}{\mathbb{F}}

\newcommand{\Ee}{\mathbb{E}}
\renewcommand{\Re}{\mathrm{Re}}
\newcommand{\cov}{G}

\newcommand{\Pro}{\mathbb{P}}
\newcommand{\Qq}{\mathbb{Q}}
\newcommand{\bfbeta}{\uple{\beta}}
\newcommand{\bfmu}{\uple{\mu}}

\DeclareMathOperator{\res}{res}

\newcommand{\expect}{\mathbb{E}}

\newcommand{\mods}[1]{\,(\mathrm{mod}\,{#1})}
\newcommand{\ov}[1]{\overline{#1}}
\newcommand{\KL}{\mathcal{K\ell}}

\newcommand{\uple}[1]{{\boldsymbol{{#1}}}}

\newcommand{\nmom}{\mathrm{M}}
\newcommand{\lra}{\longrightarrow}
\newcommand{\ra}{\rightarrow}
\newcommand{\eps}{\epsilon}
\DeclareMathOperator{\hypk}{Kl_2}
\newcommand{\rZ}{\mathrm{Z}}
\newcommand{\rand}[1]{\mathrm{{#1}}}

\begin{document}

\title[Gaussian distribution and Hecke eigenvalues]{Gaussian
  distribution for the divisor function and Hecke eigenvalues in
  arithmetic progressions}

\date{\today}

\author[\'E. Fouvry]{\'Etienne Fouvry}
\address{Universit\'e Paris Sud, Laboratoire de Math\'ematique\\
  Campus d'Orsay\\ 91405 Orsay Cedex\\France}
\email{etienne.fouvry@math.u-psud.fr}

\author[S. Ganguly]{Satadal Ganguly}
\address{Indian Statistical Institute, Kolkata, India}
\email{satadalganguly@gmail.com}

\author[E. Kowalski]{Emmanuel Kowalski}
\address{ETH Z\"urich -- D-MATH\\
  R\"amistrasse 101\\
  CH-8092 Z\"urich\\
  Switzerland} \email{kowalski@math.ethz.ch}

\author[Ph. Michel]{Philippe Michel}
\address{EPFL/SB/IMB/TAN, Station 8, CH-1015 Lausanne, Switzerland }
\email{philippe.michel@epfl.ch}

\thanks{Ph. M. was partially supported by the SNF (grant
  200021-137488) and the ERC (Advanced Research Grant 228304);
  \'E. F. thanks ETH Z\"urich, EPF Lausanne and the Institut
  Universitaire de France for financial support; S.G. thanks EPF
  Lausanne for financial support. }

\subjclass[2010]{11F11, 11F30, 11T23, 11L05,  60F05}

\keywords{Divisor function, Hecke eigenvalues, Fourier coefficients of
  modular forms, arithmetic progressions, central limit theorem,
  Kloosterman sums, monodromy group, Sato-Tate equidistribution}

\begin{abstract}
  We show that, in a restricted range, the divisor function of
  integers in residue classes modulo a prime follows a Gaussian
  distribution, and a similar result for Hecke eigenvalues of
  classical holomorphic cusp forms. Furthermore, we obtain the joint
  distribution of these arithmetic functions in two related residue
  classes. These results follow from asymptotic evaluations of the
  relevant moments, and depend crucially on results on the
  independence of monodromy groups related to products of Kloosterman
  sums.
\end{abstract}

\maketitle


\section{Introduction}

The distribution of arithmetic functions in arithmetic progressions is
one of the cornerstones of modern analytic number theory, with a
particular focus on issues surrounding uniformity with respect to the
modulus (see~\cite{F50} for a recent survey).  Besides the case of
primes in arithmetic progressions, much interest has been devoted to
the divisor function $d(n)$ and higher-divisor functions, in
particular because -- in some precise sense -- a good understanding of
a few of these is equivalent to knowledge about the primes themselves
(see, e.g.,~\cite[Th\'eor\`eme 4]{F84}).
\par
The consideration of the second moment for primes $p\leq X$ in
arithmetic progressions to moduli $q\leq Q\leq X/(\log X)^A$ leads to
the Barban-Davenport-Halberstam theorem (see,
e.g.,~\cite[Th. 17.2]{IK}), which has been refined to an asymptotic
formula for $Q=X$ by Montgomery~\cite{montg}. Similarly,
Motohashi~\cite{moto} evaluated asymptotically the variance of the
divisor function $d(n)$ for $n\leq X$ in arithmetic progressions
modulo $q\leq X$.
\par
We will show that one can determine an asymptotic distribution for the
divisor function $d(n)$ for $n\leq X$ in arithmetic progressions
modulo a single prime $p$, provided however that $X$ is a bit smaller
than $p^{2}$.

\begin{theorem}[Central Limit Theorem for the divisor function]\label{th-1-1}
  Let $w$ be a non-zero real-valued smooth function on $\Rr$ with
  compact support in $]0,+\infty[$ and with $L^2$ norm $\|w\|$.  For a
  prime $p$, let
$$
S_d(X,p,a)=\sum_{\substack{n\geq 1\\ n\equiv a\bmod p}}
{d(n)w\Bigl(\frac{n}{X}\Bigr)},
$$
and
\begin{align}
M_d(X,p)&=\frac{1}{p}\sum_{n\geq1}{d(n)w\Bigl(\frac{n}{X}\Bigr)}-
\frac{1}{p^{2}}\int_{0}^{+\infty}
(\log x+2\gamma-2\log p)w\Bigl(\frac{x}{X}\Bigr)dx\label{eq-main-term}\\
&=
\frac{1}{p }\sum_{n\geq1}{d(n)
w\Bigl(\frac{n}{X}\Bigr)}+O\Bigl(\frac1{p^{2}}X(\log X)\Bigr),\nonumber
\end{align}
where $\gamma$ is the Euler constant. For $a\in\Fpt$, let
$$
E_d(X,p,a)=\frac{S_d(X,p,a)-M_d(X,p)}{(X/p)^{1/2}}.
$$
\par 
Let $\Phi(x)\geq 1$ be any real-valued function, such that
\begin{equation*}
\Phi(x)\lra +\infty\text{ as } x\ra +\infty,\quad\quad
\Phi(x)=O_\eps ( x^{\eps}),
\end{equation*}
for any $\eps>0$ and $x\geq 1$. For any prime $p$, let
$X=p^2/\Phi(p)$. Then as $p\ra +\infty$ over prime values, the random
variables
$$
a\mapsto \frac{E_{d}(X,p,a)}{\|w\|\sqrt{2\pi^{-2}(\log \Phi(p))^3}}
$$
on $\Fpt$, with the uniform probability on $\Fpt$, converge in
distribution to a standard Gaussian with mean $0$ and variance $1$,
i.e., for any real numbers $\alpha<\beta$, we have
$$
\frac{1}{p-1}\Bigl|\Bigl\{ a\in\Fpt\,\mid\, \alpha\leq
\frac{E_{d}(X,p,a)} {\|w\|\sqrt{2\pi^{-2}(\log \Phi(p))^3}} \leq
\beta \Bigr\}\Bigr|\ \underset{p\rightarrow \infty} {\longrightarrow}\
\frac{1}{\sqrt{2\pi}}\int_{\alpha}^{\beta}e^{-t^2/2}dt.
$$
\end{theorem}


In fact, our results are more general, in three directions: (1) we
will consider, in addition to the divisor function, the Fourier
coefficients of any classical primitive holomorphic modular form $f$
of level $1$ (e.g., the Ramanujan $\tau$ function); (2) we will
compute the moments of the corresponding random
variables and, for a fixed moment, obtain a meaningful asymptotic in a
wider range of $X$ and $p$; (3) we will also consider the joint
distribution of 
$$
a\mapsto (E_{d}(X,p,a),E_d(X,p,\gamma(a)))
$$
when $\gamma$ is a fixed projective linear transformation (e.g.,
$\gamma(a)=a+1$, $\gamma(a)=2a$, $\gamma(a)=-a$, $\gamma (a)=1/a$,
which illustrate various interesting phenomena.) For all these
results, the crucial ingredients are the Voronoi summation formula,
and the Riemann Hypothesis over finite fields, in the form of results
of independence of monodromy groups of sheaves related to Kloosterman
sums.
\par
We now introduce the notation to handle these more general problems.
As in the statement above, we fix a non-zero smooth function $w\,: \
\Rr\rightarrow \Rr$, with compact support in $[w_0, w_1]$ with $0<
w_0<w_1 < +\infty$. For any modulus $c\geq 1$, let
$$
S_d(X,c,a)=\sum_{\substack{n\geq 1\\ n\equiv a\bmod c}}
{d(n)w\Bigl(\frac{n}{X}\Bigr)}.
$$
\par
This sum has, asymptotically, a natural main term (see,
e.g.,~\cite{La-Zh}) which we denote by $M_d(X,c)$, and which coincides
with $M_d(X,p)$ when $c=p$ is prime (see \eqref{eq-voronoi-d} below).
The number of terms in $S_d(X,c,a)$ is $\approx X/c$ and the
\emph{square root cancellation philosophy} suggests that its
difference with the main term should be of size
\begin{equation}\label{srp}
  S_d(X,c,a)-M_d(X,c)\ll (X/c)^\frac{1}{2} X^\epsilon,
\end{equation}
as long as $X/c$ gets large. Thus the map
$$
\rZ\, :\ a\in (\Zz/c\Zz)^\times \mapsto
E_d(X,c,a)=\frac{S_d(X,c,a)-M_d(X,c)}{(X/c)^{1/2}}.
$$ 
is a natural normalized error term that we wish to study as a random
variable on $(\Zz/c\Zz)^\times$ equipped with the uniform probability
measure (here and below, we sometimes omit the dependency on $p$ and
$X$ to lighten the notation $\rZ$).
\par
Similarly, consider a primitive (Hecke eigenform) holomorphic cusp
form $f$ of even weight $k$ and level $1$ (these restrictions are
mainly imposed for simplicity of exposition). We write
$$ 
f ( z ) = \sum_{n \geq 1} \rho_f ( n ) n^{(k-1)/2}e ( n z )
$$
its Fourier expansion at infinity, so that $\rho_f(1)=1$ and
$\rho_f(n)$ is the eigenvalue of the Hecke operator $T(n)$ (suitably
normalized). We let
\begin{gather*}
  S_f ( X, c, a ) =\sum_{n \equiv a ( \textnormal{mod } c )} \rho_f
  (n) w\Bigl(\frac{n}{X}\Bigr),\quad M_f(X, c) =\frac{1}{c} \sum_{n
    \geq 1}\rho_f ( n )w\Bigl(\frac{n}{X}\Bigr),\\
  E_f(X, c, a) = \frac{S_f ( X, c, a )-M_f(X, c)}{(X/c)^{1/2}},
\end{gather*}
for $c\geq 1$ and any integer $a$. Note that, in this case, the
integral   representation
$$
M_f(X,c)=\frac{1}{c}\times \frac{1}{2\pi
  i}\int_{2-i\infty}^{2+i\infty} \hat w (s)\,X^s\,L(s, f ) ds
$$
in terms of the Mellin transform $\hat{w}$ of $w$ shows that the main
term is very small, namely
\begin{equation}\label{aze} 
  M_f(X,c)
  \ll_{f,A} c^{-1}X^{-A}
\end{equation}  for every positive $A$, 
uniformly for $c\geq 1$ and $X \geq 1$.
\par
We will study the distribution of
$$
a\mapsto  E_f(X, p, a),\quad\quad a\mapsto E_d(X,p,a)
$$
for $p$ prime using the method of moments. Thus, for any integer
$\kappa\geq 1$, we define
\begin{equation}\label{defM} 
{\mathcal M}_{\star} (X,c\,;\kappa)=\frac{1}{c}\sum_{\substack{a\bmod c\\(a,c)=1}}
E_{\star}(X, c, a)^\kappa,\quad\quad \star=d\text{ or } f.
\end{equation}
\par
The first moment is very easy to estimate, and besides Motohashi's
work (which considers the average of $\mathcal{M}_d(X,c;2)$ over
$c\leq X$), the second moment has recently been discussed by
Blomer~\cite{Bl}, L\"u~\cite{Lu} and Lau--Zhao~\cite{La-Zh}. In
particular, Lau and Zhao obtained an asymptotic formula in the range
$X^{1/2}<c<X$ (see \eqref{19:56} below; note that the range
$c<X^{1/2}$ seems to be much more delicate.)
\par
We will evaluate \emph{any} moment, in a suitable range. Precisely, in
\S \ref{ProofTHM1.1} we will prove:

\begin{theorem}\label{funda}
  Let the notation be as above, with $\star=d$, the divisor function,
  or $\star=f$, $f$ a Hecke form of weight $k$ and level $1$. Let $p$
  be a prime number. Then, for every integer $\kappa\geq 1$, for every
  positive $\delta $, for every positive $\epsilon$, for every $X$
  satisfying
\begin{equation}\label{X1/2<p<X}
2\leq X^{1/2}\leq p < X^{1-\delta},
\end{equation}
  we have the equality 
\begin{equation}\label{mark}
{\mathcal M}_{\star} (X,p\,;\kappa) = C_{\star} (\kappa) 
+
O\Bigl(
p^{-1/2+\eps}\Bigl(\frac{p^2}X\Bigr)^{\kappa/2}
+\Bigl(
\frac{X}{p^2}\Bigr)^{1/2+\eps}\Bigr),
\end{equation}
where the implied constant depends on $(\delta, \eps,\kappa, f,w )$,
and the constant $C_{\star}(\kappa)$ is given by
\begin{equation}\label{eq-cfk}
  C_{\star}(\kappa)=c_{\star,w}^{\kappa/2} \,m_{\kappa},
\end{equation}
with
\begin{equation}\label{mkappa}
m_{\kappa}=
\begin{cases}
  0 & \text{ if $\kappa$ is odd,}\\
  \\
  \displaystyle{\frac{\kappa!}{2^{\kappa/2}(\kappa/2)!}}&\text{ if
    $\kappa$ is even},
\end{cases}
\end{equation}
and 
\begin{gather}\label{eq-mk}
  c_{f,w}=\|w\|^2\|f\|^2 \frac{(4\pi)^{k}}{\Gamma(k)},\quad\quad
  c_{d,w}=P_w\Bigl(\log \frac{p^2}{X}\Bigr),
\end{gather}
for some polynomial $P_w(T)\in\Rr[T]$, depending only on $w$, of
degree $3$ with leading term $2\|w\|^2\pi^{-2}T^3$. Here, for a cusp
form $f$, the $L^2$-norm of $f$ is computed with respect to the
probability measure
$$
\frac{3}{\pi}\frac{dxdy}{y^2}
$$
on $\mathrm{SL}_2(\Zz)\backslash \mathbb{H}$, and the $L^2$-norm of
$w$ is computed with respect to the Lebesgue measure on $\Rr$.
\end{theorem}

\begin{remark}
  In the case $\kappa=2$, and in the range $X^{1/2}\leq c\leq X$, Lau
  and Zhao~\cite[Theorem 1 (2)]{La-Zh} have obtained
\begin{equation}\label{19:56}
  \frac{1}c\sum_{a=1}^c \,\Bigl\vert    
  \frac{c^{1/2}}{X^{1/2}}\sum_{\substack{n \equiv a ( \textnormal{mod
      } c ) 
      \\1 \leq n \leq X}} \rho_f ( n ) 
  \Bigr\vert^2 =c_{f} + 
  O\Bigl(\Bigl(\frac{c}{X}\Bigr)^\frac{1}{6}d (c)+
  \Bigl(\frac{X}{c^2}\Bigr)^{\frac{1}{4}} 
  \sum_{\ell \mid c} \frac{\varphi (\ell)}{\ell}\Bigr),
\end{equation}
for any modulus $c\geq 1$ (not only primes), and a similar result for
the divisor function.
\end{remark}

We will make further comments on this result after the proof, in
Section~\ref{ssec-comments}. Since $m_{\kappa}$ is the $\kappa$-th
moment of a Gaussian random variable with mean $0$ and variance $1$,
we obtain the following, which implies Theorem~\ref{th-1-1} in the
case $\star=d$:

\begin{corollary}[Central limit theorem]
\label{cor-gaussian}
Let $\Phi(x)\geq 1$ be any real-valued function, such that
\begin{equation*}
\Phi(x)\lra +\infty\text{ as } x\ra +\infty,\quad\quad
\Phi(x)=O_\eps ( x^{\eps}),
\end{equation*}
for any $\eps>0$, uniformly for $x\geq 1$. For any prime $p$, let
$X=p^2/\Phi(p)$. Then as $p\ra +\infty$ over prime values, the random
variables
$$
a\mapsto \frac{E_{\star}(X,p,a)} {\sqrt{c_{\star,w}}}
$$
on $\Fpt$ converge in distribution to a standard Gaussian with mean
$0$ and variance $1$.
\end{corollary}

As far as we know, this is the first result of this type. We will
prove this in Section \ref{ProofTHM1.1}, and give further comments, in
Section~\ref{sec-gaussian}. 
\par
Among the natural generalizations of this result, we consider next the
following one: given a map $a\mapsto \gamma(a)$ on $\Fpt$, what is the
asymptotic joint distribution of
$$
a\mapsto (E_{\star}(X,p,a),E_{\star}(X,p,\gamma(a)))\ ?
$$
\par
We study this when $\gamma$ is given by a fractional linear
transformation. Precisely, let 
\begin{equation}\label{defgamma}
\gamma=
\begin{pmatrix}a&b \\c&d
\end{pmatrix}\in \mathrm{GL}_2(\Qq)\cap M_2(\Zz)
\end{equation}
be a fixed invertible matrix with integral coefficients.  For $p\nmid
\det \gamma$, the matrix $ \gamma$ has a canonical reduction modulo
$p$ in ${\rm PGL}_2 (\Fp)$, which we denote by $\pi_p (\gamma)$.  In
the usual manner, $\gamma$ (or $\pi_p (\gamma)$) defines a fractional
linear transformation on $\Pro^1_{\Fp}$ by
$$
z\in  \Pro^1_{\Fp}\mapsto \gamma \cdot z=\frac{az+b}{cz+d}.
$$ 
\par
By Corollary \ref{cor-gaussian}, we know that, in the range of
validity of this result, both 
\begin{equation}\label{defX}
  \rZ\,:\,a\mapsto \frac{E_{\star}(X,p,a)} {\sqrt{c_{\star,w}}}\ \text{
    and } \rZ\circ\gamma\,:\,
  a \mapsto  \frac{E_{\star}(X,p,\gamma \cdot a)} {\sqrt{c_{\star,w}}}, 
\end{equation}
seen as random variables defined on the set
$$
\{ a\in {\mathbb F}_p \ \vert\ a,
\gamma \cdot a\not=0,\infty\} 
$$
converge to the normal law. We then wish to know the asymptotic joint
distribution of the vector $(\rZ,\rZ\circ \gamma)$, and we study this
issue, as before, using moments.
\par
For $\kappa$ and $\lambda$ positive integers,
let
\begin{equation}\label{defMmix} 
{\mathcal M} _\star(X,p\,; \kappa ,
\lambda\,; \gamma):=\frac1p \sum_{\substack{a\in {\mathbb F}_p \\
    a,\ \gamma \cdot a\not=0,\infty}} E_\star (X,p,a)^{\kappa }\,
E_\star (X,p,\gamma\cdot a)^{\lambda},
\end{equation}
be the {\it mixed moment of order } $(\kappa, \lambda)$. 
\par
In analogy with Theorem~\ref{funda}, we will estimate these moments in
\S \ref{PROOFMIXED}. To state the result, we note that if $\gamma$ is
diagonal, there is a unique triple of integers $(\alpha_{\gamma},
\gamma_1, \gamma_{2})$, such that we have the {\it canonical form}
\begin{equation}
\label{canonical}
\gamma=\alpha_\gamma\begin{pmatrix}\gamma_1& 0\\
0& \gamma_{2}
\end{pmatrix}, \ \gamma_1 \geq 1 \text{ and } (\gamma_{1},\gamma_2)=1.
\end{equation}
\par
We further introduce the arithmetic functions
\begin{equation}\label{cold}
  \uple{\rho}_{a,f}=\prod_{p^\alpha \Vert a} \Bigl( \rho_f
  (p^\alpha) - \frac{\rho_{f}(p)\rho_{f}(p^{\alpha
      -1})}{p+1}\Bigr),\quad\quad
  \uple{\rho}_{a,d}=\prod_{p^\alpha \Vert a} \Bigl( d(p^\alpha) -
  \frac{d(p)d(p^{\alpha-1})}{p+1}\Bigr),
\end{equation}
for $a\geq 1$, and $\uple{\rho}_{a,f}=0$ for $a<0$,
$\uple{\rho}_{a,d}=\uple{\rho}_{-a,d}$ for $a<0$. For $\star=f$, we
also define the constant
\begin{equation}\label{defcf}
c_f= \|f\|^2 (4\pi)^k\Gamma(k)^{-1}.
\end{equation}

Our result is:

\begin{theorem}\label{MIXEDMOMENTS}  
  Let $\gamma$ be defined by~\emph{(\ref{defgamma}}).
\par
\emph{(1)} For every integers $\kappa$ and $\lambda$, for every
$\delta $ and $\eps >0$, for every prime $p\geq p_0 (\gamma)$ and $X $
satisfying~\emph{(\ref{X1/2<p<X})}, there exists
$C_{\star}(\kappa,\lambda,\gamma)$ such that
\begin{equation}\label{eq-mixed-conv}
  {\mathcal M} _\star(X,p\,;\kappa , \lambda\,; \gamma)= 
  C_{\star}(\kappa,\lambda,\gamma)
  + O\Bigl( p^{-\frac{1}{2} +\epsilon}\Bigl(
  \frac{p^2}{X}\Bigr)^{(\kappa+\lambda)/2} 
  +\Bigl(\frac{X}{p^2}\Bigr)^{1/2+\eps} \Bigr).
\end{equation}
\par
\emph{(2)} If $ \gamma $ is non-diagonal, then
\begin{equation}\label{nondiag}
 C_{\star}(\kappa,\lambda,\gamma)=C_{\star}(\kappa) C_{\star}(
 \lambda).
\end{equation}
\par
\emph{(3)} If $ \gamma $ is diagonal, and written in the canonical
form~\emph{(\ref{canonical})}, then
\begin{equation}\label{diag}
C_{\star}(\kappa,\lambda,\gamma)=
\begin{cases}
  0\text{ if } \kappa+\lambda \text{ is odd,}
  \\
  \displaystyle{\sum_{\substack{0\leq \nu \leq \min(\kappa,\lambda)\\
        \nu\equiv \kappa \equiv \lambda\bmod 2}}} \nu!
  \binom{\kappa}{\nu}
  \binom{\lambda}{\nu}m_{\kappa-\nu}m_{\lambda-\nu}
  \,(c_{\star,w})^{\frac{\kappa+\lambda}{2}-\nu}\, (\tilde c_{\star,
    w,\gamma})^\nu,\text{ otherwise,}
\end{cases}
\end{equation}
where
$$
\tilde c_{f,w,\gamma}=c_f\uple{\rho}_{\gamma_{1}\gamma_{2}, f}\Bigl(
\int_{-\infty}^\infty w(\gamma_{1}t) w(\gamma_{2}t) dt\Bigr),
$$
and for $\star=d$, we have
$$
\tilde c_{d, w,\gamma}= P_{\gamma_{1}\gamma_{2},w}\Bigl( \log
\frac{p^2}{X}\Bigr)
$$
for some polynomial $P_{\gamma_{1}\gamma_{2},w}(T)\in\Rr[T]$, of
degree $\leq 3$ and with coefficient of $T^3$ given by
$$
\frac{2}{\pi^2}\uple{\rho}_{\gamma_{1}\gamma_{2},d}\Bigl(
\int_{-\infty}^\infty w(\gamma_{1}t) w(\gamma_{2}t) dt\Bigr)T^3.
$$
\par
In~\emph{(\ref{eq-mixed-conv})}, the implied constant depends at most
on $(\gamma, \delta, \varepsilon,\kappa,\lambda)$, and
in~\emph{(\ref{diag})}, we make the convention that $0^{\nu}=1$ if
$\nu =0$.
\end{theorem}

Of course, if $\gamma$ is the identity, we recover Theorem
\ref{funda}.  More generally, we can now determine the joint
asymptotic distribution of $(\mathrm Z,\mathrm Z \circ \gamma)$ in the
same range as Corollary~\ref{cor-gaussian}.

\begin{corollary}\label{gaussianvector}
  Let $\Phi$ be a function as in Corollary \ref{cor-gaussian}, and let
  $X=p^2/\Phi (p)$.  Then, for $\star =f $ or $d$, as $p$ tends to
  infinity, the random vector $(\mathrm Z, \mathrm Z \circ \gamma)$
  converges in distribution to a centered Gaussian vector with
  covariance matrix
\begin{gather}
\begin{pmatrix} 1 &  0\\
  0&1
\end{pmatrix},\quad\quad\text{ if } \gamma \text{ is not diagonal}.
\\
\begin{pmatrix} 1 & \cov_{\star,\gamma,w}\\
  \cov_{\star,\gamma,w}&1
\end{pmatrix},\quad\quad\text{ if } \gamma \text{ is diagonal},
\label{eq-cor-diag}
\end{gather}
where the covariance $\cov_{\star,\gamma,w}$ is given by
$$
\cov_{\star,\gamma,w}=\frac{\rho_{\gamma_{1}\gamma_{2},\star}}{\|w\|^2}
\int_{\Rr}{w(\gamma_{1}t)w(\gamma_{2}t)dt}.
$$
\end{corollary}

Recall that a pair $(\mathrm X,\mathrm Y)$ of random variables is a
{\it Gaussian vector } if and only, for every complex numbers $\alpha$
and $\beta$, the random variable $\alpha \mathrm X+\beta \mathrm Y$
has a Gaussian distribution (see, e.g., \cite[pp. 121--124]{JaPr}). If
$(\mathrm X,\mathrm Y)$ is a Gaussian vector, its covariance matrix
$\mathrm {cov}(\mathrm X, \mathrm Y)$ is defined by
\begin{equation}\label{covXY}
  \mathrm {cov}(\mathrm X, \mathrm Y)=
 \begin{pmatrix}
   \Ee (\mathrm X^2)-\Ee(\mathrm{X})^2 &\Ee (\mathrm X\mathrm
   Y)-\Ee(\mathrm X)\Ee(\mathrm Y) \\
   \Ee (\mathrm X\mathrm Y)-\Ee(\mathrm X)\Ee(\mathrm Y)& \Ee (\mathrm
   Y^2)-\Ee(\mathrm Y)^2
 \end{pmatrix},
\end{equation}
where $\Ee$ denotes the expectation of a random variable. Recall also
that a Gaussian vector $(\mathrm X, \mathrm Y) $ has independent
components if and only if $\Ee (\mathrm X \mathrm Y)= \Ee (\mathrm X)
\Ee (\mathrm Y)$, i.e., if the covariance matrix is diagonal (see
\cite[Theorem 16.4]{JaPr} for instance). Thus from Corollary
\ref{gaussianvector} (noting that $\uple{\rho}_{a,d}\not=0$ for any
integer $a\not=0$), we get a criterion for asymptotic independence of
$(\rZ,\rZ\circ \gamma)$:

\begin{corollary}\label{final}  
  We adopt the notations and hypotheses of Corollary
  \ref{gaussianvector}.  Then as $p$ tends to $\infty$, the random
  variables $\mathrm Z$ and $\mathrm Z \circ \gamma$ tend to
  independent Gaussian random variables, if and only if one of the
  following conditions holds:
\par
\emph{(1)} If $\gamma$ is not a diagonal matrix, i.e., $a\mapsto
\gamma\cdot a$ is not a homothety,
\par
\emph{(2)} If $\gamma$ is a diagonal matrix and $$\int_{-\infty
}^\infty w(\gamma_{1}t) w(\gamma_{2}t) \, dt =0,$$
\par
\emph{(3)} If $\star=f$, $\gamma$ is a diagonal matrix in the
from~\emph{(\ref{canonical})}, and there exists a prime $p$ and
$\alpha\geq 1$ such that $p^\alpha \Vert \gamma_{2}\gamma_{1}$ and
such that
$$
(p+1) \rho_{f}(p^\alpha) = \rho_{f }(p) \rho_{f }(p^{\alpha -1}).
$$
\end{corollary}
 
\begin{remark} 
  (1) Corollary \ref{final} shows for instance that, for $p\ra
  \infty$, the random variables $a \mapsto E_{\star}(p^2/\Phi (p),
  p,a)$ and $a \mapsto E_{\star}(p^2/\Phi (p), p, \gamma \cdot a)$
  converge to independent Gaussian variables, if $\gamma$ is one of
  the following functions
$$
\gamma \cdot a= a+1,\quad\quad
\gamma \cdot a =-a,\quad\quad 
\gamma \cdot a=1/a.
$$
\par
The case of $\gamma\cdot a= 2a$ is more delicate, since it depends on
the value of the integral $\int_{0}^{+\infty } w(t) w(2t) \, dt$. For
instance, this integral is zero when one has the inequalities $w_{0}<
w_{1}<2w_{0}<2w_{1}$, where as before $\mathrm{supp}(w)\subset [w_{0},
w_{1}]$. The possible dependency here reflects the obvious fact that
if $n\equiv a\bmod{p}$ and $d\mid n$, then $2n\equiv 2a\bmod{p}$ and
$d\mid 2n$.
\par
(2) We do not know if any primitive Hecke form $f$ of level $1$ exists
for which Condition (3) in this last corollary holds for some
$p^\alpha$!  Certainly the ``easiest'' way it could apply would be if,
for some $p$, we had $\rho_f(p)=0$, but the existence of a primitive
cusp form of level $1$ and a prime $p$ with $\rho_f(p)=0$ seems
doubtful (e.g., a conjecture of Maeda suggests that the characteristic
polynomials of the Hecke operators $T(p)$ in level $1$ are
irreducible.) On the other hand, if we extend the result to forms of
fixed level $N\geq 1$, it is possible to have $\rho_f(p)=0$ for some
$p$ (e.g., for weight $k=2$ and $f$ corresponding to an elliptic
curve.)
\end{remark}

\subsection{Sketch of the proof}

We will sketch the proof in the case of cusp forms, which is
technically a bit simpler, though we present the actual proofs in a
unified manner.  For Theorem~\ref{funda}, the crucial starting point  is the
Voronoi summation formula, as in~\cite{Bl,La-Zh}, which expresses
$E_f(X, c,a)$ for any $c\geq 1$ in terms of sums weighted by some
smooth function of the Fourier coefficients $\rho_f(n)$ twisted by
Kloosterman sums $S(a,n;c)$. One then sees that the main
contribution to this sum comes from the $n$ of size roughly $Y=c^2/X$
(see Proposition \ref{stack1}).
\par
Considering the $\kappa$-th moment, we obtain therefore an average
over $a\bmod p$ of a product of $\kappa$ Kloosterman sums $S(a,
n_{i};p)$, where all variables $n_i$ are of size approximately
$p^2/X$. The sum over $a\in \Fpt$, when the  variables $n_{i}$ are fixed, can be
evaluated using deep results on the independence of Kloosterman
sheaves (see Proposition~\ref{390}). This allows us to gain a factor
$p^{1/2}$ compared with a direct application of the Weil bound for
Kloosterman sums, except for special, well-understood, configurations
of the $n_i$ modulo $p$. These configurations lead, by combinatorial
arguments, to the Gaussian main term of Theorem~\ref{funda}.  (Note
that we can take no advantage of the summation over the variables
$n_i$, which turn out to have a short range in the cases where our
result is non-trivial, see Section~\ref{ssec-comments}.)

The study of mixed moments (see Theorem \ref{MIXEDMOMENTS}) has a lot
of similarities.  The only significant difference lies in the study of
the independence of Klosterman sheaves, when some of them are twisted
by the rational transformation $\gamma$.  However, Proposition
\ref{390} is general enough to show that these sheaves are dependent
if and only if we are in the ``obvious'' cases.  The main terms then
require some computations of integrals using properties of the Bessel
transforms.

\subsection{Possible extensions} 

A Gaussian law similarly appears if one studies the random variable
$a\mapsto E_\star(X,p, P(a))$, where $P$ is a non--constant fixed
polynomial with integer coefficients. The fact that $P$ is not
necessarily a bijection on $\Ff_p$ does not affect the Gaussian
behavior. The proof of this extension requires a suitable
generalization of Proposition \ref{390}. 

It also seems that the present method can be extended to the study of
the distribution of sums of the shape
$$
a\mapsto S_\star(X,p,K_a)=\sum_{n\geq
  1}\tau_\star(n)K_a(n)w\Bigl(\frac{n}X\Bigr)
$$
where $\tau_\star(\cdot)$ is either $d(\cdot)$ or $\rho_f(\cdot)$, and
$K_a(n)=K(an)$ for a fairly general trace function $K$ as
in~\cite{FKM}. The shape of the analogue of Theorem \ref{funda} would
then depend on the nature of the geometric monodromy group of a
suitable ``Bessel transform'' of the sheaves underlying $K(\cdot)$.

Another natural extension, which we are currently considering, is that
of coefficients of cusp forms on higher-rank groups, and of higher
divisor functions.
  
\subsection{Notations} 

We use synonymously the notation $f(x)\ll g(x)$ for $x\in X$ and
$f=O(g)$ for $x\in X$. We denote $e(z)=e^{2i\pi z}$ for $z\in\Cc$. For
$c\geq 1$ and $a$, $b$ integers, or congruence classes modulo $c$, the
Kloosterman sum $S(a,b;c)$ is defined by
$$
S(a,b;c)=\sum_{\substack{x\bmod
    c\\(x,c)=1}}{e\Bigl(\frac{ax+b\bar{x}}{c}\Bigr)}
$$
where $\bar{x}$ is the inverse of $x$ modulo $c$. The normalized
Kloosterman sum is defined by
$$
\hypk(a,b;c)=\frac{S(a,b;c)}{c^{1/2}},
$$
and for $(a,b,c)=1$ it satisfies the Weil bound
\begin{equation}\label{eq-weil}
  |\hypk(a,b;c)|\leq d(c).
\end{equation}
To lighten notations, we define 
$$
\hypk(a;c):=\hypk(a,1;c),
$$
and recall the equality $\hypk (a,b;c)=\hypk (ab;c)$, whenever $(b,c)=1.$
\par
We will use the Bessel functions $J_{k-1}$, where $k\geq 2$ is an
integer, $Y_0$ and $K_0$; precise definitions can be found for
instance in~\cite[App. B.4]{Iwvert} and in \cite{Wa08}.


\section{Preliminaries}

We gather in this section some facts we will need during the proof of
the main results.  The reader may wish to skip to
Section~\ref{ProofTHM1.1} and refer to the results when they are
needed.
\par
We begin with the Voronoi formula in the form we need:
  
\begin{proposition}[Voronoi summation]\label{stack1}
  Let $\star=f$, for a cusp form $f$ of level $1$ and weight $k$, or
  $\star=d$. Let $c$ be any positive integer, with $c$ prime if
  $\star=d$. Then for any $X \geq 1$ and for any integer $a$, we have
  the equality
\begin{equation}\label{basic}
E_{\star}(X, c, a) = \frac{X^{1/2}}{c} \sum_{\substack{c_1 |c\\c_1 > 1}} 
\Bigl(\frac{c}{c_1}\Bigr)^{1/2} \sum_{n\not=0}
\tau_{\star}(n) W_{\star} \Bigl(\frac{nX}{c_1^2}\Bigr)
\hypk(  a,  n; c_1 ),
\end{equation}
where $n$ runs on the right over non-zero integers in $\Zz$ and
\begin{align}
\tau_{f}(n)&=\begin{cases}
\rho_f(n)&\text{ if } n\geq 1,\\
0&\text{ otherwise,}
\end{cases}\label{tauf}
\\
\tau_d(n)&=d(|n|),
\end{align}
and
\begin{align}\label{Bessel1}
  W_f(y) &=2\pi i^k \int_0^\infty w(u)J_{k-1}(4\pi \sqrt{ uy}) d
  u\quad
  \text{ for } y>0,\\
  W_f(y)&=0,\quad \text{ for } y<0,\nonumber\\
  W_d(y)&=-2\pi \int_0^\infty w(u)Y_0(4\pi \sqrt{ uy}) d u,\quad
  \text{ for } y>0,\label{Bessel2}\\
  W_d(y)&=4 \int_0^\infty w(u)K_0(4\pi \sqrt{ u|y|}) d u,\quad
  \text{ for } y<0.\label{Bessel3}
\end{align}
\par
In particular, if $c=p$, a prime, we have
\begin{equation}\label{basicp}
  E_{\star}(X, p,a ) = \Bigl(\frac{X}{p^2}\Bigr)^{1/2}  \sum_{  n
   \not= 0} \tau_{\star} ( n ) W_{\star}
  \Bigl(\frac{nX}{p^2}\Bigr) \hypk ( a,  n; p ).
\end{equation}
\end{proposition}

For the proof we recall the standard Voronoi summation formula (see,
e.g., \cite[p.~83]{IK} for $\star=f$ and~\cite[(4.49)]{IK} for
$\star=d$, which we rewrite as a single sum over positive and negative
integers instead of two sums).

\begin{lemma}\label{Voronoi1} 
  Let $c$ be a positive integer and $a$ an integer coprime to $c$. 
\par
\emph{(1)} For any smooth function $w$ compactly supported on $]0,
\infty[$, we have
$$
\sum_{n\geq 1} \rho_f(n)w(n)e\Bigl(\frac{an}{c}\Bigr) = 
\frac{1}{c} \sum_{n\geq
  1}\rho_f(n)W_f\Bigl(\frac{n}{c^2}\Bigr) 
e\Bigl(-\frac{n\overline{a}}{c}\Bigr),
$$
if $f$ is a cusp form of level $1$ and weight $k$.
\par
\emph{(2)} For any smooth function $w$ compactly supported on $]0,
\infty[$, we have
\begin{equation}
  \sum_{n\geq 1} d(n)w(n)e\Bigl(\frac{an}{c}\Bigr) = 
\frac{1}{c} \int_0^{+\infty}{(\log
    x+2\gamma-2\log c)w(x)dx}
+  \frac{1}{c} \sum_{n\not=0}d(|n|)W_d \Bigl(\frac{n}{c^2}\Bigr)
  e\Bigl(-\frac{\overline{a}n}{c}\Bigr).\label{eq-voronoi-d}
\end{equation}
\end{lemma}

\begin{proof}[Proof of Proposition~\ref{basic}]
  We consider the case of $\star=f$, the divisor function being
  handled similarly (it is easier since $c$ is prime; the
  definition~(\ref{eq-main-term}) of the main term is designed to
  cancel out the first main term in~(\ref{eq-voronoi-d})).  Using
  orthogonality of additive characters, and separating the
  contribution of the trivial character from the others, we write
\begin{align*}
  S_{f}(X,c, a ) &= \frac{1}{c} \sum_{b = 0}^{c - 1} e \Bigl(
  -\frac{ab}{c}\Bigr) \sum_{n \geq 1}
  \rho_{f}(n) w\Bigl(\frac{n}{X}\Bigr) e \Bigl( \frac{bn}{c} \Bigr)\\
  &= \frac{1}{c} \sum_{n \geq 1} \rho_f ( n ) w\Bigl(\frac{n}{X}\Bigr)
  + \frac{1}{c} \sum_{1\leq b \leq c-1} e \Bigl(-\frac{ab}{c}\Bigr)
  \sum_{n \geq 1} \rho_f ( n ) w\Bigl(\frac{n}{X}\Bigr) e \Bigl(
  \frac{bn}{c} \Bigr),
 \end{align*} 
 which yields the expression
$$
E_f(X, c, a) =\frac{1}{(cX)^{1/2}} \sum_{1\leq b \leq c-1} e \Bigl(
-\frac{ab}{c} \Bigr) \sum_{n \geq 1} \rho_f ( n )
w\Bigl(\frac{n}{X}\Bigr) e \Bigl( \frac{bn}{c} \Bigr).
$$
\par
We split the second according to the value of the g.c.d $d=(b,c)$,
writing 
$$
d=( b, c ), b = d b_1, c = d c_1,
$$
and note that 
$$
1 < c_1 \leq c, \quad 1 \leq b_1 < c_1 . 
$$
\par
We then get
\begin{align*}  
  E_f(X, c, a) &= \frac{1}{(cX)^{1/2}} \sum_{d \mid c} \sum_{\substack{1 \leq b
      < c \\( b, c ) = d}} e \Bigl( -\frac{ab}{c}\Bigr) \sum_{n \geq
    1} \rho_f ( n ) w\Bigl(\frac{n}{X}\Bigr) e \Bigl(
  \frac{bn}{c}  \Bigr) \\
  &= \frac{1}{(cX)^{1/2}} \sum_{\substack{c_1 |c\\ c_1 > 1}} \sum_{\substack{1
      \leq b_1 < c_1\\ ( b_1, c_1 ) = 1}} e \Bigl( - \frac{a b_1}{
    c_1} \Bigr) \sum_{n \geq 1} \rho_f ( n ) w\Bigl(\frac{n}{X}\Bigr)
  e \Bigl(\frac{b_1n}{c_1} \Bigr).
\end{align*}
\par
We can now apply Lemma \ref{Voronoi1} since $(b_1,c_1)=1$, and we get
$$ 
\sum_{n \geq 1} \rho_f ( n ) w\Bigl(\frac{n}{X}\Bigr) e \Bigl(
\frac{b_1n}{c_1} \Bigr) 
= \frac{X}{c_1}
\sum_{n \geq 1} \rho_f ( n ) W\Bigl(\frac{nX}{c_1^2}\Bigr) e \Bigl( -\frac{
\overline{b}_1n}{c_1} \Bigr).
$$
\par
The proposition now follows since the terms with $n<0$ are identically
zero for this case.
\end{proof}

We will need some basic information on the behavior of the Bessel
transforms $W_{\star}(y)$.

\begin{proposition}\label{TTT}
  Let $w$ be a smooth function with support included in $]0,
  +\infty[$.  Let $W_{\star}(y)$ be one of the Bessel transforms of
  $w$ as defined in Proposition~\ref{stack1}, for some integer $k\geq
  2$ in the case $\star=f$ of weight $k$.
\par
\emph{(1)} The function $W_{\star}$ is smooth on $\Rr^{\times}$, and
for every $A\geq 0$ and $j\geq 0$, we have
\begin{equation}\label{weak}
  y^j W_{\star}^{(j)}(y) \ll_{A,j}
  \min \bigl(1 +  \bigl\vert \log \vert y\vert  \bigr\vert , |y|^{-A}\bigl),
\end{equation}
for $y\not=0$.
\par
\emph{(2)} We have
\begin{equation}\label{eq-hankel}
  \|W_{\star}\|=\|w\|,
\end{equation}
where the $L^2$-norm of $W_{\star}$ and $w$ are computed in
$L^2(\Rr^{\times})$ with respect to Lebesgue measure.
\par
\emph{(3)} More generally, for any two non-zero real numbers $m$ and
$n$, we have 
$$
\int_{-\infty}^\infty W_{\star}(mt) W_{\star}(nt) dt=
\int_{-\infty}^\infty w(mt) w(nt) dt.
$$
\end{proposition}

\begin{proof}
  (1) (Compare, e.g., with~\cite[p. 280]{Bl}, \cite[Lemma 3.1]{La-Zh})
  We begin with the case $j=0$.  For $y$ small, we use the bounds
$$ 
J_{k-1} ( x ) \ll_{k} 1,\quad Y_0(x)\ll 1+\vert \log x\,\vert ,\quad
K_0(x)\ll 1+ \vert \log x\,\vert 
$$
for $0<x\leq 1$ which immediately imply
\begin{equation}\label{285}
W_{\star}(y)\ll 1+ \bigl\vert  \log \vert y\vert \bigr\vert
\end{equation}
in all cases.
\par
To deal with the case where $|y|\geq 1$, we first make the change of
variable
$$
v=4\pi\sqrt{u|y|}
$$
in the integrals \eqref{Bessel1} (resp.~(\ref{Bessel2}),
(\ref{Bessel3})), so that we always get
$$
W_{\star}(y) =\frac{1}{|y|}\int_0^\infty
w\Bigl(\frac{v^2}{16\pi^2y^2}\Bigr) v B_{0}(v) d v,
$$
where $B_0=c J_{k-1}$, $0$,  $cY_0$ or $cK_0$, for some fixed multiplicative
constant $c\in\Cc$.
\par
We denote $\alpha=(16\pi^2y^2)^{-1}$. To exploit conveniently the
oscillations of the Bessel functions $B_0$ we integrate by parts,
using the relations (see~\cite[8.472.3, 8.486.14]{gr})
\begin{equation}\label{deriv1}
  ( x^{\nu + 1} Z_{\nu + 1} ( x ) )' = \eps x^{\nu + 1} Z_{\nu} ( x
  ),\quad\quad
\end{equation}
where
$$
\eps =\begin{cases}
+1 &\text{ if }  Z_\nu = J_\nu \text{ or } Y_\nu,\\
\\
-1 &\text{ if } Z_\nu = K_\nu.
\end{cases}
$$
For $\star=f$, remembering that $w$ vanishes at $0$ and $\infty$, we
obtain, for instance, the equality 
$$
W_{f}(y)=-\frac{c}{|y|} \int_0^\infty \Bigl(2\alpha v^2 w'(\alpha
v^2) +(1-k) w(\alpha v^2)\Bigr)J_k(v) d v \ \ (y>0).
$$
\par
By iterating $\ell\geq 1$ times, and then arguing similarly for
$\star=d$, we see that there exist coefficients $\xi_{\ell, \nu}$ such
that
\begin{equation}\label{dark}
  W_{\star}(y)= \frac{1}{|y|}
  \int_0^\infty\Bigl(\, \sum_{\nu=0}^\ell \xi_{\ell, \nu} \ 
  (\alpha v^2)^\nu  w^{(\nu)}(\alpha v^2)\Bigr) v^{-\ell +1}\, 
  B_{\ell} (v)d v,
\end{equation}
where $B_{\ell}=J_{k-1+\ell}$, $ 0$, $Y_{\ell}$ or $K_{\ell}$
corresponding to the different cases $\star=f$ or $\star =d$, $y>0$ or
$y<0$.
\par
Since $w$ has compact support in $[w_{0}, w_{1}]$, the above integral
can be restricted to the interval
$$
{\mathcal I}:=\bigl[(w_0/ \alpha)^{1/2}, (w_1/ \alpha)^{1/2}\bigr], $$
and using the estimates\footnote{\ For the last one, one knows in fact
  that $K_{\ell}(v)$ decays exponentially fast for $v\ra +\infty$.}
$$
J_{k-1+\ell } (v) \ll_{\ell} v^{-1/2},\quad\quad Y_{\ell}(v)\ll_{\ell}
v^{-1/2},\quad\quad K_{\ell}(v)\ll_{\ell} v^{-1/2}
$$
for $v\geq 1$, we obtain the inequality 
\begin{equation}\label{314}
  W_{\star}(y)  \ll
  |y|^{-1}
  \int_{\mathcal I} v^{-\ell +\frac{1}{2}} d
  v \ll |y|^{-1-\ell/2 +3/2}
\end{equation}
for $|y|\geq 1$. Since $\ell\geq 0$ is arbitrary, this gives the
result for $j=0$.
\par
We can reduce the general case to $j=0$ using the formulas
(see~\cite[8.472.2, 8.486.13]{gr})
$$
xZ_{\nu}'(x)=\nu Z_{\nu}(x)-xZ_{\nu+1}(x),
$$
from which it follows that
$$
y\frac{d}{dy}\Bigl(\int_{0}^{\infty}w(u)Z_{\nu}(4\pi\sqrt{uy})du\Bigr)=
\frac{\nu}{2}\int_{0}^{\infty}w(u)Z_{\nu}(4\pi\sqrt{uy})du-
2\pi\sqrt{y}\int_{0}^{\infty}w(u)\sqrt{u}Z_{\nu+1}(4\pi\sqrt{uy})du.
$$
\par
Applying the previous method to the relevant Bessel functions then
leads to 
$$
yW_{\star}'(y)\ll_{\star,A} \min(1+|\log|y||,y^{-A})
$$
and by induction a similar argument deals with higher derivatives.
\par
(2) In the case $\star=f$, the identity
$$
\int_0^{+\infty}W_f(u)^2du= \int_0^{+\infty}{w(u)^2du}=\|w\|^2
$$
is a direct consequence of the unitarity of the Hankel transform,
i.e., of the Fourier transform for radial functions (see,
e.g.,~\cite[Lemma 3.4]{La-Zh}). The case $\star=d$ is less classical,
although it is formally similar, the hyperbolas $xy=r$ replacing the
circles $x^2+y^2=r^2$ (see~\cite[\S 4.5]{IK}). We use a
representation-theoretic argument to get a quick proof. The unitary
principal series representation $\rho=\pi(0)$ of $\mathrm{PGL}_2(\Rr)$
(in the notation of~\cite[p. 10]{c-ps}) can be defined by its Kirillov
model with respect to the additive character $\psi(x)=e(x)$, which is
a unitary representation of $\mathrm{PGL}_2(\Rr)$ on
$L^2(\Rr^{\times}, |x|^{-1}dx)$. In this model, the
unitary operator
$$
T=\rho\Bigl(\begin{pmatrix}0&-1\\1&0
\end{pmatrix}\Bigr)
$$
on $L^2(\Rr^{\times},|x|^{-1}dx)$ is given by
$$
T\varphi(x)=\int_{\Rr^{\times}}{\varphi(t)\mathcal{J}(xt)\frac{dt}{|t|}},
$$
where $\mathcal{J}$ is the so-called ``Bessel function'' of $\rho$
(with respect to $\psi$,
see~\cite[Th. 4.1]{c-ps}). By~\cite[Prop. 6.1, (ii)]{c-ps} (see
also~\cite[\S 6, \S 21]{baruch-mao}), we have
$$
\mathcal{J}(u)=
\begin{cases}
-  2\pi \sqrt{u}Y_0(4\pi\sqrt{u})&\text{ for } u>0,\\
  4\sqrt{|u|}K_0(4\pi\sqrt{|u|})&\text{ for } u<0.\\
\end{cases}
$$
\par
Hence by~(\ref{Bessel2}) and~(\ref{Bessel3}), we see that
\begin{equation}\label{eq-kirillov}
W_d(y)=|y|^{-1/2}T(\varphi)(y),\quad
\text{ where } \quad
\varphi(x)=
\begin{cases}
\sqrt{x}w(x)&\text{ if } x>0\\
0&\text{ if } x<0.
\end{cases}
\end{equation}
\par
The unitarity of $T$ means that
$$
\int_{\Rr^{\times}}{|T(\varphi)(y)|^2\frac{dy}{|y|}}=
\int_{\Rr^{\times}}{|\varphi(x)|^2\frac{dx}{|x|}},
$$
i.e.,
$$
\int_{\Rr^{\times}}{|W_d(y)|^2dy}=
\int_{0}^{+\infty}{|w(x)|^2dx}=\|w\|^2.
$$
\par
(3) We consider different cases.  If $mn>0$, changing $t$ to $-t$
allows us to assume that $m$ and $n$ are positive. Then a simple
polarization argument from \eqref{eq-hankel} shows that
\begin{equation}\label{1691}
  \int_{-\infty}^{+\infty}{W_{\star}(mt)W_{\star}(nt)dt}
  = \int_{-\infty}^\infty    {\mathfrak w}_{m} (u)  {\mathfrak w}_{n} ( u)du,
\end{equation}
where $u\mapsto {\mathfrak w}_{m}(u) $ is the function for which the
Bessel transform of is $t\mapsto W_{\star}(m t)$ and similarly for
$\mathfrak{w}_n(u)$. But it is immediate that ${\mathfrak w}_{m} (u) =
(1/m) \,w (u/m)$, and therefore \eqref{1691} gives the result.
\par
If $mn<0$, then since the support of $w$ is contained in
$[0,+\infty[$, we have $w(mt)w(nt)=0$ for all $t$, hence
$$
\int_{\Rr}{w(mt)w(nt)dt}=0,
$$
and we must show that the integral of $W_{\star}(mt)W_{\star}(nt)$ is
also zero. If $\star = f$, a cusp form, this is immediate since
$W_f(y)=0$ for $y<0$, so that $W_f(mt)W_f(nt)=0$ for all $t$.  
\par
For $\star=d$, we use representation theory as in (2). With the same
notation as used there, and for any real-number $a\not=0$, we denote
$$
U_a=\rho\Bigl(\begin{pmatrix}a&0\\0&1
\end{pmatrix}
\Bigr)
$$
so that, by definition of the Kirillov model (see~\cite[\S 4.2,
(4.1)]{c-ps}), we have
$$
U_a(\varphi)(x)=\varphi(ax)
$$
for $\varphi\in L^2(\Rr^{\times},|x|^{-1}dx)$. Observe that, in
$\mathrm{PGL}_2(\Rr)$, we have
$$
\begin{pmatrix}0&-1\\1&0
\end{pmatrix}
\begin{pmatrix}a&0\\0&1
\end{pmatrix}=
\begin{pmatrix}-1&0\\0&-a
\end{pmatrix}
\begin{pmatrix}0&-1\\1&0
\end{pmatrix}=
\begin{pmatrix}a^{-1}&0\\0&1
\end{pmatrix}
\begin{pmatrix}0&-1\\1&0
\end{pmatrix},
$$
hence
$$
T\circ U_a=U_{a^{-1}}\circ T.
$$
\par
Using this and the unitarity of $T$, we deduce that
\begin{align*}
  \int_{\Rr}{(T\varphi)(ax)\overline{(T\varphi)(bx)}\frac{dx}{|x|}} =
  \langle U_a(T\varphi),U_b(T\varphi)\rangle &= \langle
  T(U_{a^{-1}}\varphi),T(U_{b^{-1}}\varphi)\rangle\\&=\langle
  U_{a^{-1}}\varphi, U_{b^{-1}}\varphi\rangle\\&=
  \int_{\Rr}{\varphi\Bigl(\frac{x}{a}\Bigr)
\overline{\varphi\Bigl(\frac{x}{b}\Bigr)}\frac{dx}{|x|}}
=\int_{\Rr}{\varphi(bx)\overline{\varphi(ax)}\frac{dx}{|x|}}.
\end{align*}
\par
Now, applying~(\ref{eq-kirillov}) and the fact that $W_d$ is
real-valued, we derive
$$
\int_{\Rr}{W_d(ax)W_d(bx)dx}=\int_{\Rr}{w(ax)w(bx)dx}
$$
for all non-zero $a$ and $b$.
\end{proof}

\begin{remark}
  One can also give a direct proof of the last part of this
  proposition using known properties of Bessel functions: the crucial
  point is that the function
$$
\psi(a,b)=\int_0^{\infty}Y_0(a\sqrt{y})K_0(b\sqrt{y})\,dy
$$
is antisymmetric, which follows from an explicit evaluation
using~\cite[6.523]{gr} and~\cite[p. 153, 2.34]{ober}. Conversely, the
results for cusp forms can be proved using representation theory, the
discrete series representation of weight $k$ replacing the
representation $\rho$.
\end{remark}

Our last preliminary results concern the sums which will give rise to
the leading terms in the main results.  Recall the definitions and
\eqref{defcf}.

\begin{proposition}\label{pr-mainterm}
  Let $p$ be a prime number, $\delta>0$ a parameter and $X\geq 1$ such
  that
$$
X^{1/2}\leq p\leq X^{1-\delta}.
$$
\par
Let $Y=p^2/X$. For $\star\in \{d,f\}$, and for $a$ and $b$ coprime
non-zero integers, not necessarily positive, let
$$
\mathcal B_{\star} (a,b,Y)= \sum_{\substack{n\not=0\\1\leq \vert
    an\vert,\, \vert bn \vert < p/2}} \tau_{\star} (an)
\tau_{\star}(bn)W_{\star}\Bigl( \frac{an}{Y}\Bigr) W_{\star}\Bigl(
\frac{bn}{Y}\Bigr).
$$
\par
\emph{(1)} If $\star=f$, we have
$$
\mathcal B_{\star} (a,b,Y)= c_f\uple{\rho}_{ab,f}\Bigl(
\int_{-\infty}^\infty w(a t) w(b t)dt \Bigr) Y+ O( Y^{1/2+\eps})
$$
for any $\eps>0$.
\par
\emph{(2)} If $\star=d$, there exists a polynomial $P_{ab}\in \Rr[T]$
of degree at most $3$, depending on $w$, such that
$$
\mathcal{B}_d(a,b,Y)= P_{ab} (\log Y)Y +O (Y^{\frac 12 +\eps})
$$
for any $\eps>0$, and with coefficient of $T^3$ given by
\begin{equation}\label{eq-leading}
\frac{2}{\pi^{2}}\uple{\rho}_{ab,d}\Bigl( \int_{-\infty}^\infty w(a t)
w(b t) \, dt \Bigr)T^3.
\end{equation}
\par
In both cases, the implied constants depend on
$(\delta,\eps,\star,a,b)$. 
\end{proposition}

We will use standard complex integration techniques, and first
determine the relevant generating series (it is here that it is
important that $f$ be a Hecke eigenform.) We denote
$$
F_{\star}(s)=\sum_{n\not=0}\tau_{\star}(n)^2|n|^{-s},
$$
so that
$$
F_{f}(s)=\frac{L(s,f\times f)}{\zeta(2s)}
$$
if $f$ is a Hecke eigenform, where $L(s,f\times f)$ is the
Rankin-Selberg convolution $L$-function, and
$$
F_d(s)=2\frac{\zeta(s)^4}{\zeta(2s)}.
$$
\par
In both cases, $F_{\star}(s)$ extends to a meromorphic function, with
polynomial growth in vertical strips, for $\Re(s)>1/2$. It has only a
pole at $s=1$ in this region (of order $1$ if $\star=f$, and order $4$
if $\star=d$).

\begin{lemma}\label{lm-generate}
  Let $\star=f$ or $d$, and $a$, $b$ be non-zero coprime integers. Let
$$
F_{\star,a,b}(s)=\sum_{n\not=0}{\tau_f(an)\tau_f(bn)|n|^{-s}}.
$$
\par
If $\star=f$ and $ab<0$, we have $F_{\star,a,b}=0$. Otherwise, we have
$$
F_{\star,a,b}(s)=F_{\star}(s)\prod_{p^{\nu_p}\mid\mid ab}
\Bigl(\tau_f(p^{\nu_p})-\frac{\tau_f(p)\tau_f(p^{\nu_p-1})}{p^s+1}\Bigr).
$$
\par
In particular, $F_{\star,a,b}$ always extends to a meromorphic
function for $\Re(s)>1/2$, with polynomial growth in vertical strips.
\end{lemma}

\begin{proof}
  One sees immediately that it is enough to treat the case where $a$,
  $b\geq 1$ and $ab\not=1$. Then the assumption that $(a,b)=1$ allows
  us to write
$$
F_{\star,a,b}(s)=F_{\star,ab,1}(s)
$$
so that we can further reduce to the case where $b=1$, in which case
we write $F_{\star,a,1}=F_{\star,a}$. Now, writing any integer
$n\not=0$ (uniquely) as $n=jm$ where $j\geq 1$ has all prime factors
dividing $a$ and $m\not=0$ is coprime with $a$, and summing over $j$
first, we get
\begin{align*}
F_{\star,a}(s)&=\sum_{1\leq j\mid a^{\infty}}\sum_{(m,a)=1}
\tau_f(jm)\tau_f(ajm)|jm|^{-s}\\
&=\sum_{j\mid a^{\infty}}{\tau_f(j)\tau_f(aj)j^{-s}}
\sum_{(m,a)=1}\tau_f(m)^2|m|^{-s}\\
&=F_{\star}(s)\Bigl(\prod_{p\mid a}\sum_{k\geq
  0}\tau_f(p^k)^2p^{-ks}\Bigr)^{-1}
\sum_{j\mid a^{\infty}}\tau_f(j)\tau_f(aj)j^{-s},
\end{align*}
by multiplicativity of $\tau_f$.
\par
Now write
$$
a=\prod_{p\mid a}{p^{\nu_p}}
$$
the factorization of $a$. Again by multiplicativity, we get
$$
\sum_{j\mid a^{\infty}}\tau_f(j)\tau_f(aj)j^{-s}=
\prod_{p\mid a}\sum_{k\geq 0}{\tau_f(p^k)\tau_f(p^{k+\nu_p})p^{-ks}}. 
$$
\par
Let
$$
G_i=\sum_{k\geq 0}{\tau_f(p^k)\tau_f(p^{k+i})p^{-ks}}
$$
for some fixed prime $p$ and integer $i\geq 0$. For $i\geq 1$ and
$k\geq 1$, we have
$$
\tau_f(p^{k+i})=\tau_f(p^k)\tau_f(p^i)-\tau_f(p^{k-1})\tau_f(p^{i-1)},
$$
and therefore
$$
G_i=\tau_f(p^i)G_0-p^{-s}\tau_f(p^{i-1})G_1
$$
for $i\geq 1$. In particular, the case $i=1$ gives
$$
(1+p^{-s})G_1=\tau_f(p)G_0,
$$
which then implies that
$$
G_i=\Bigl(\tau_f(p^i)-\frac{\tau_f(p)\tau_f(p^{i-1})}{p^s+1}\Bigr)G_0
$$
for $i\geq 1$. Now, since $\nu_p\geq 1$ by definition, it follows that
$$
F_{\star,a}(s)=F_{\star}(s)\prod_{p\mid a}
\Bigl(\tau_f(p^{\nu_p})-\frac{\tau_f(p)\tau_f(p^{\nu_p-1})}{p^s+1}\Bigr)
$$
as claimed.
\end{proof}

\begin{proof}[Proof of Proposition~\ref{pr-mainterm}]
Using Proposition~\ref{TTT}, (1), we obtain first 
$$
\mathcal{B}_{\star}(a,b,Y)=\mathcal{B}^0_{\star}(a,b,Y)+O(p^{-1})
$$
where
$$
\mathcal{B}^0_{\star}(a,b,Y)=\sum_{n\not=0}{\tau_{\star}(an)\tau_{\star}(bn)
  W_{\star}\Bigl(\frac{an}{Y}\Bigr)W_{\star}\Bigl(\frac{bn}{Y}\Bigr)}.
$$
\par
Let $\varphi$ be the Mellin transform of
$$
x\mapsto
W_{\star}\Bigl(\frac{ax}{Y}\Bigr)W_{\star}\Bigl(\frac{bx}{Y}\Bigr). 
$$
\par
Again by Proposition~\ref{TTT}, (1), this is a holomorphic function,
bounded and decaying quickly in vertical strips, for $\Re(s)>0$. We
have the integral representation
$$
\mathcal{B}^0_{\star}(a,b,Y)=\frac{1}{2i\pi}
\int_{(2)}F_{\star,ab}(s)Y^s\varphi(s)ds,
$$
and we proceed to shift the contour to $\Re(s)=1/2+\eps$, for a fixed
$\eps>0$. The integral on the line $\Re(s)=1/2+\eps$ satisfies
$$
\frac{1}{2i\pi}
\int_{(1/2+\eps)}F_{\star,ab}(s)Y^s\varphi(s)ds
\ll Y^{1/2+\eps}
$$
where the implied constant depends on $(\star, a,b,\eps,w)$. On the
other hand, the unique singularity that occurs during the shift of
contour is the pole at $s=1$. 
\par
If $\star=f$, this is a simple pole and
$$
\res_{s=1}F_{\star,ab}(s)Y^s\varphi(s)=
Y\varphi(1)\res_{s=1}F_{f,ab}(s).
$$
\par
Since
$$
\varphi(1)=\int_{\Rr}{W_f(at)W_f(bt)dt}=\int_{\Rr}{w(at)w(bt)dt}
$$
by Proposition~\ref{TTT}, (3), and since it is well-known that
$$
\res_{s=1}F_{f}(s)=\|f\|^2(4\pi)^{k}\Gamma(k)^{-1}=c_f,
$$
(from Rankin-Selberg theory, see, e.g., \cite[(13.52), (13.53)]{Iw}),
we see that Lemma~\ref{lm-generate} exactly gives the result in the
case of a cusp form
\par
If $\star=d$, on the other hand, we have a pole of order $4$, and we
see that
$$
\res_{s=1}F_{d,ab}(s)Y^s\varphi(s)= YP_{ab}(\log Y)
$$
where the polynomial $P_{a,b}$ has degree at most $3$ and has
coefficient of $T^3$ given by
$$
\frac{1}{6}\frac{1}{\zeta(2)}\uple{\rho}_{ab,d}
\Bigl(\int_{\Rr}{W_d(at)W_d(bt)dt}\Bigr)T^3= \frac{2}{\pi^2}
\uple{\rho}_{ab,d}\Bigl(\int_{\Rr}{w(at)w(bt)dt}\Bigr)T^3,
$$
again by Proposition~\ref{TTT}, (3). This concludes the proof.
\end{proof}

\section{Proof of Theorem \ref{funda}}\label{ProofTHM1.1}

\subsection{First step}\label{ssec-first-step}

Let $p$ be a prime such that the condition \eqref{X1/2<p<X} holds. To
shorten the notation, we write
\begin{equation}\label{defY}
Y= p^2/X,
\end{equation}
which is $\geq 1$ under our assumption. We also write simply
$W=W_{\star}$ depending on whether we treat the case of cusp forms or
of the divisor function.
\par
From \eqref{basicp} in Proposition~\ref{stack1}, we deduce 
\begin{multline}\label{358*}
{\mathcal M}_{\star} (X,p;\kappa)=\frac{1}{pY^{\kappa/2}}
\ \underset{n_{1},\dots, n_{\kappa}\not=0}{\sum \cdots \sum}
\tau_{\star}(n_{1})\cdots \tau_{\star}(n_{\kappa})
W\Bigl(\frac{n_{1}}{Y}\Bigr)\cdots W\Bigl(\frac{n_{\kappa}}{Y}\Bigr)\\
\times \sum_{ 1\leq a < p}\hypk(a n_{1};p)\cdots
\hypk(a n_{\kappa};p),
\end{multline}
which we write in the form
\begin{equation}\label{367}
{\mathcal M}_{\star} (X,p;\kappa):=\frac{1}{pY^{\kappa/2}}\bigl(
\Sigma_{1}+\Sigma_{2}\bigr)
\end{equation}
where $\Sigma_{1}$ corresponds to the contribution of the $(n_{1},
\dots, n_{\kappa})$ such that $1\leq |n_{i}|< p/2$ for all $i$ and
$\Sigma_{2}$ is the complementary contribution of those $(n_{1},
\dots, n_{\kappa})$ such that $ |n_{i}|\geq p/2$ for one $i$ at least.

\subsection{Study of $\Sigma_{2}$}\label{ssec-sigma2}

We first deal with $\Sigma_2$, which is easy. By symmetry, we may
restrict to the case where $|n_{1}|\geq p/2$. By Deligne's bound
\begin{equation}\label{Del}
|\rho_f(n)|\leq d(n)
\end{equation}
(in the case of a Hecke eigenform $f$) and the Weil
bound~(\ref{eq-weil}) for Kloosterman sums, we have in both cases
$$
\Sigma_{2}\ll \Bigl(\sum_{|n_{1}|\geq p/2}
d(|n_{1}|)\, \Bigl| W\Bigl(\frac{n_1}{Y}\Bigr)
\Bigr|\Bigr)\times  \Bigl(\sum_{n\not=0}
d(|n|)\,\Bigl| W \Bigl(\frac{n}{Y}\Bigr)\Bigr|\Bigr)^{\kappa-1}.
$$
\par
Applying~(\ref{weak}) with $A\geq 3$, we deduce
$$
\Sigma_{2}\ll \ X^\epsilon\,(Y^A/p^{A-1})
  Y ^{\kappa-1}
$$
for any $\epsilon>0$ and hence
\begin{equation*}
  \Sigma_2\ll
  X^{\epsilon}p\, \Bigl(\frac{p}{X}\Bigr)^{2A}\, 
  \Bigl(\frac{p^2}{X}\Bigr)^{\kappa -1}.
\end{equation*}
\par
By assumption, we have  $p<X^{1-\delta}$, hence taking $A=A(\delta, \kappa)$
sufficiently large we  prove the inequality
\begin{equation}\label{Sigma2}
\Sigma_2\ll X^{-1},
\end{equation} 
which combined with \eqref{367}
is acceptable in view of the error term claimed in
\eqref{mark}. 

\subsection{Study of $\Sigma_{1}$}\label{4.1} 

The study of $\Sigma_1$ is the crux of the matter. To handle precisely
the sum of Kloosterman sums over $a$ in~(\ref{358*}), which is a sum
over a finite field, we will use a deep result in algebraic
geometry. But first of all, we must prepare the combinatorial
configurations of the arguments $n_1$, \ldots, $n_{\kappa}$, in order
to be able to detect the main term. We shall even put it in a more
general setting to cover the proof of Theorem \ref{MIXEDMOMENTS}. The
following definition deals with the {\it decreasing sequence of
  multiplicities}.

\begin{definition}[Configuration]\label{defconf}
  Let $p$ be prime. Let $\bfbeta :=(\beta_{1},\dots,
  \beta_{\kappa})\in(\PGL_2(\Fp))^{\kappa}$ be a $\kappa$-tuple of
  projective linear transformations modulo $p$. There exist an integer
  $\nu$ satisfying $1\leq \nu\leq \kappa$, a $\nu$-tuple
  $\uple{\mu}=(\mu_{1},\dots, \mu_{\nu})$ of positive integers
  $\mu_{i}$ satisfying
$$
\mu_{1}\geq  \mu_{2}\geq  \cdots  \geq 
\mu_{\nu}\geq 1\text{ and } \mu_{1}+\cdots +\mu_{\nu}=\kappa.
$$
and $\nu$ distinct elements $(\sigma_{1},\cdots,\sigma_\nu)\in
\bigl(\PGL_2(\Fp)\bigr)^\nu$, such that we have
$$
\bigl\{\beta_{1},\cdots, \beta_{\kappa}
\bigr\}=\{ \sigma_{1},\cdots , \sigma_{\nu}\},
$$
and
$$
|\{i\, :\, 1\leq i \leq \kappa,\, \beta_{i} =\sigma_{j}\}| =\mu_{j}
$$
for all $j$, with $1\leq j\leq \nu$. The integer $\nu$ and the
$\nu$-tuple $(\mu_{1},\dots, \mu_{\nu})$ are unique, and the latter
will be called the \emph{configuration} of $\bfbeta$, the integer
$\nu$ will be called the \emph{length of the configuration} and the
entries $\mu_j$ its \emph{ multiplicities}.
\par
If all the multiplicities $\mu_{j}$ are even, we will say that
$\bfbeta$ has a \emph{mirror configuration}.  In particular its length
$\mu$ is even.
\end{definition}

In the next proposition, we will see that the asymptotics for a sum of
products of Kloosterman sums shifted by the projective transformations $\beta_i$
depends only on the configuration of $\boldsymbol \beta $, rather than
on the precise values of the $\beta_i $.

\begin{proposition}\label{390} 
  Let $p$ be a prime.  Let $\kappa \geq 1$, ${\boldsymbol
    \beta}=(\beta_{1},\dots,\beta_{\kappa})\in \bigl(
  \PGL_2(\Fp)\bigr)^\kappa$ be a $\kappa$-tuple of elements of the
  projective linear group with associated configuration ${\boldsymbol
    \mu}=(\mu_{1}, \dots, \mu_{\nu})$.
\par
Consider the sum
$$
{\mathfrak S} (\kappa,{\boldsymbol \beta},p)= \sum_{\substack{a\bmod
    p\\ \beta_i\cdot a\not=0,\infty (1\leq i\leq \kappa)}}
\hypk(\beta_{1}\cdot a ;p)\cdots \hypk(\beta_{\kappa}\cdot a ;p).
$$
\par
We then have
\begin{equation}\label{579}
{\mathfrak S} (\kappa,{\boldsymbol \beta} ,p)= 
A({\boldsymbol \mu})p+O_\kappa(p^\frac{1}{2}),
\end{equation}
where $A({\boldsymbol \mu})$ is the product of integrals
$$
A({\boldsymbol \mu})=
\Bigl(\,\frac{2}{\pi}\int_{0}^{\pi}(2\cos\theta)^{\mu_{1}}\,\sin^2
\theta d \theta\Bigr)\cdots
\Bigl(\,\frac{2}{\pi}\int_{0}^{\pi}(2\cos\theta)^{\mu_{\nu}}\,\sin^2
\theta d \theta\Bigr).
$$
\par
The product $A({\boldsymbol \mu})$ is an integer, which is positive if
and only if $\boldsymbol \beta $ is in a mirror configuration and $0$
otherwise, in which case we have
$$
{\mathfrak S} (\kappa,{\boldsymbol \beta} ,p)= O(p^\frac{1}{2}).
$$ 
\par
Finally we have
\begin{equation}\label{valueA}
  A(2,2,\dots,2) = 1.
\end{equation}
\end{proposition}

 This is a generalization of a result of Fouvry, Michel, Rivat and
  S\'ark\"ozy (see~\cite[Lemma 2.1]{FMRS}), which only dealt with the
  case where the $\beta_{i}$ are all diagonal and distinct modulo $p$.

\begin{proof}
  By the definition of the configuration, the sum equals
$$ 
{\mathfrak S} (\kappa,{\boldsymbol \beta},p)=
\sum_{\substack{a\in\Fp,\ \sigma_i\cdot a\not=0,\infty\\ (1\leq i\leq
    \nu)}}\hypk(\sigma_1\cdot a ;p)^{\mu_1}\cdots \hypk (\sigma_\nu
\cdot a ;p)^{\mu_\nu},
$$
where the elements $\sigma_i$, $1\leq i\leq \nu$, are distinct in
$\PGL_2(\Fp)$.
\par
For $\ell\not= p$, let $\KL$ be the (normalized) $\ell$-adic
Kloosterman sheaf constructed by Deligne and studied by Katz
in~\cite{Katz-gskm}. This is a lisse $\ov\Qq_\ell$-sheaf of rank $2$
on $\Gg_{m,\Fp}$, which has trivial determinant. For some isomorphism
$\iota\,:\,\ov\Qq_{\ell}\rightarrow \Cc$, it satisfies
$$
\iota(\mathrm{trace}(\mathrm{Frob}_{a,\Fp}|\KL))=-\hypk (a;p)
$$
for any $a\in\Fp^\times$. Moreover, $\KL$ is Lie-irreducible, tamely
ramified at $0$ with a single unipotent Jordan block, and wildly
ramified at $\infty$ with Swan conductor $1$ and with a single break
at $1/2$.
\par
Given $\gamma\in\PGL_2(\Fp)$, let $\gamma^*\KL$ be the pullback of
$\KL$ by the fractional linear transformation $\gamma:\ x \mapsto
\gamma\cdot x$; this sheaf is lisse on
$\mathbb{P}^1_{\Fp}-\{\gamma^{-1}(\{0,\infty\})\}$ and for any
$a\in\Fp$ such that $\gamma \cdot a\not=0,\infty$, it satisfies
$$
\iota(\mathrm{trace}(\mathrm{Frob}_{a,\Fp}|\gamma^*\KL))
=-\hypk(\gamma\cdot a ;p).
$$
\par
Katz~\cite{Katz-gskm} computed the geometric monodromy group of
$\KL$, and showed that it is equal to $\SL_2$, and coincides with the
arithmetic monodromy group of $\KL$. The same is therefore true for
$\gamma^*\KL$.

We make the following:
\begin{claim}
  For $\sigma_1$ and $\sigma_2$ distinct elements of $\PGL_2(\Fp)$ and
  $\mscL$ any rank one sheaf, lisse on some non-empty open subset of
  $\mathbb{P}^1_{\Fp}$, the sheaves $\sigma_1^*\KL\otimes\mscL$ and
  $\sigma_2^*\KL$ are not geometrically isomorphic.
\end{claim}

\begin{proof}
  We may assume that $\sigma_1=\mathrm{Id}$ and that $\sigma=\sigma_2$
  is not the identity. If $\sigma$ is an homothety, the claim was
  proven in \cite[Lemme 2.4]{Mi}. We now reduce to this case. Assume
  that $\KL\otimes\mscL$ and $\sigma^*\KL$ are geometrically
  isomorphic. Since $\mscL$ is of rank $1$, its only possible breaks
  at infinity are integer, and hence $\KL\otimes\mscL$ is wildly
  ramified at $\infty$. So $\sigma^*\KL$ is also wildly ramified at
  infinity, which means that $\sigma\cdot\infty=\infty$. Furthermore,
  $\KL\otimes\mscL$ is also ramified at $0$, and hence $\sigma^*\KL$
  must also be ramified, which means $\sigma\cdot 0=0$. But this
  implies that $\sigma$ is a homothety, and we apply the result
  of~\cite{Mi}.
\end{proof}

Since the $\sigma_i,\ (i=1,\cdots ,\nu)$ are distinct elements in
$\PGL_2(\Fp)$, it follows from the Goursat-Kolchin-Ribet criterion
(see~\cite[Prop. 1.8.2]{katz-esde}) that the geometric monodromy group
of the tensor product
$$
\sigma_1^*\KL\otimes\cdots\otimes\sigma_\nu^*\KL
$$
is equal to its arithmetic monodromy group and is the full product
group
$$
\SL_2\times\cdots\times\SL_2,
$$
which indicates an asymptotic independence of the values of the
Kloosterman sums $\hypk(\sigma_i \cdot a ;p)$ as $a$ varies over $\Fp$
such that $\sigma_i \cdot a\not=0,\infty$, ($i=1,\dots,\nu$).
\par
Using Katz's effective form of Deligne's equidistribution theorem
(\cite[\S 3.6]{Katz-gskm}), we deduce that
$$
\frac{1}{p-1}\sum_{\substack{a\in\Fp,\ \sigma_i.a\not=0,\infty\\
    (1\leq i\leq \nu)}}{\rm Kl\,}(\beta_{1}\cdot a,1;p)^{\mu_1}\cdots
{\rm Kl\,}(\beta_{\nu}\cdot a,1;p)^{\mu_\nu}=\prod_{i=1}^\nu
\mu_{ST}((2\cos(\theta))^{\mu_i})+O_{\mu_1,\cdots,\mu_\nu}(p^{-1/2}),
$$
where the implied constant is independent of $p$ and $\mu_{ST}$
denotes the Sato-Tate probability measure on $[0,\pi]$, which is given
by
$$
\mu_{ST}(f(\theta))=\frac{2}{\pi}\int_0^\pi f(\theta)\sin^2\theta
d \theta
$$
(recall that $[0,\pi]$ is identified with the set of conjugacy classes
of the compact group $\mathrm{SU}_2(\Cc)$ via the map
$$
g\in \mathrm{SU}_2(\Cc)\mapsto \mathrm{trace}(g)=2\cos\theta,
$$
and that the Sato-Tate measure is the image of the probability Haar
measure of $\mathrm{SU}_2(\Cc)$ under this map.)
\par
It follows by character theory of compact groups that
$$
\mathrm{mult}(\mu)=\mu_{ST}((2\cos \theta)^{\mu})
$$ 
is precisely the multiplicity of the trivial representation in the
$\mu$-th tensor power $\mathrm{Std}^{\otimes \mu}$ of the
standard $2$-dimensional representation of $\mathrm{SU}_2(\Cc)$.  In
particular, $\mathrm{mult}(\mu)$ is a non-negative integer, and it
is zero if and only if $\mu$ is odd (this is obvious when writing
the integrals; representation-theoretically, $\mathrm{mult}(\mu)=0$
if $\mu$ is odd because $\begin{pmatrix}-1&0\\0&-1
\end{pmatrix}$ acts by multiplication by $(-1)^{\mu}$ on
$\mathrm{Std}^{\otimes \mu}$, and $\mathrm{mult}(\mu)\geq 1$ for $\mu$
even, because $\mathrm{Std}^{\otimes\mu}$ is self-dual so
$\mathrm{mult}(2\mu)$ is the multiplicity of the trivial
representation in $\mathrm{End}(\mathrm{Std}^{\otimes\mu})$, and the
identity endomorphism gives an invariant subspace; in fact, one can
check that $\mathrm{mult}(2\mu)=\binom{2\mu}{\mu}/(\mu+1)$, a Catalan
number.)
\par
As a consequence
$$
A(\mu_1,\cdots,\mu_\nu)=\prod_{i=1}^\nu \mathrm{mult}(\mu_i),
$$
is a non-negative integer, and it is non-zero if and only if all the
$\mu_i$ are even, which corresponds precisely to the mirror
configuration. Since $\mathrm{mult}(2)=1$, we also have
$A(2,\cdots,2)=1$.
\end{proof}

\begin{remark}
  Expanding the Kloosterman sums, we see that
  $\mathfrak{S}(\kappa,\bfbeta,p)$ is a character sum in $\kappa+1$
  variables. The proposition shows that this character sum has
  square-root cancellation, except if $\bfbeta$ is in mirror
  configuration. As in~\cite{FKM}, we see that the \emph{structure} of
  $\mathfrak{S}(\kappa,\bfbeta,p)$ (as a sum of products of
  Kloosterman sums) is crucial to our success, since it reduces the
  problem to detecting cancellation in the single variable $a$.
\par
If $\kappa=2$ and if $\beta_1(a)=b_1a$ and $\beta_2(a)=b_2a$ are
diagonal, we can use the fact that the Kloosterman sum is the discrete
Fourier transform of the function $x\mapsto e(\bar{x}/p)$ (and
$0\mapsto 0$) to get
$$
\mathfrak{S}(2,(\beta_1,\beta_2),p)=
\sum_{a\in\Fpt}{\hypk(b_1a;p)\hypk(b_2a;p)}=
\sum_{x\in\Fpt}{e\Bigl(\frac{\bar{x}(1-\bar{b}_1b_2)}{p}\Bigr)}-\frac{1}{p}
$$
by the discrete Plancherel formula. This is essentially a Ramanujan
sum, and hence we see that the second moment (as in~(\ref{19:56}))
does not require such delicate considerations. Moreover, because the
error term is here $\ll p^{-1}$ (instead of $p^{-1/2}$), the error
term for the second moment is better than for the others, which
explains the greater range of uniformity in the formula~(\ref{19:56})
of Lau and Zhao. More generally, for $\kappa=2$ and arbitrary
$\beta_1$, $\beta_2\in\PGL_2(\Fp)$, the sum
$\mathfrak{S}(2,(\beta_1,\beta_2),p)$ can be identified with a special
case of a \emph{correlation sum} as defined in~\cite[\S 1.2]{FKM}, for
the trace weight $K(n)=e(\bar{n}/p)$. The results of~\cite[Th. 9.1, \S
11.1]{FKM} imply the statement of Proposition~\ref{390} for
$\kappa=2$.
\end{remark}

We can now continue our study of the sum $\Sigma_{1}$ defined in
\eqref{367}. Since we have $p\nmid n_i$, we have
$$
\hypk(an_{i}; p)= \hypk(\beta_{i}\cdot a;p),
$$ 
where $\beta_{i}\in  \PGL_2(\Fp)$ corresponds to the matrix  
$$
\begin{pmatrix} n_i&0\\0&1\end{pmatrix}\mods{p}.
$$
 We denote ${\boldsymbol \beta}=(\beta_{i},\dots,\beta_{\kappa})$. We also
denote by ${\boldsymbol \mu}({\boldsymbol \beta})$ the configuration of
$\boldsymbol \beta$.  Thus, by Proposition \ref{390} and by \eqref{weak}, we have the equalities
\begin{align}
  \Sigma_{1}&= p\underset{1\leq |n_{1}|,\dots, |n_{\kappa}| <p/2}{\sum
    \cdots \sum} A\bigl({\boldsymbol \mu}({\boldsymbol \beta})\bigr)\,
  \tau_{\star}(n_{1})\cdots \tau_{\star}(n_{\kappa})
  W\Bigl(\frac{n_1}{Y}\Bigr)\cdots W\Bigl(\frac{n_\kappa}{Y}\Bigr)
  \nonumber\\
  &+O\Bigl(p^{\frac{1}{2}} \Bigl(\, \sum_{1 \leq |n| < p/2} d(|n|)
  \,\Bigl| W\Bigl(\frac{n}{Y}\Bigr)\Bigr|
  \,\Bigr)^\kappa\,\Bigr)\nonumber\\
  &= 
 p \Sigma_{1,M} + O
  \bigl(\,p^{\frac{1}{2}+\eps}Y^{\kappa}\bigr),
\label{632}
\end{align}
say, for any $\epsilon>0$.

Collecting \eqref{367}, \eqref{Sigma2} and \eqref{632}, the proof of
Theorem \ref{funda} is already complete when $ \kappa$ is odd, since
trivially $\Sigma_{1,M}=0$ in that case.

\subsection{Study of $\Sigma_{1,M}$ for even $\kappa$} 

Remark that, by the definition of $\Sigma_{1}$, we have the congruence
$n_{i}\equiv n_{j}\bmod p$ if and only if $n_{i}=n_{j}.$ In the
summation over $\boldsymbol n =(n_{1},\dots, n_{\kappa})$ defining
$\Sigma_{1,M}$, we can restrict the summation over the set of
$\boldsymbol n$ such that the associated $\boldsymbol \beta$ is in
mirror configuration by Proposition~\ref{390}.

We now show that, in fact, the main contribution comes from the
$\boldsymbol n$ in mirror configuration such that the configuration of
the associated $\boldsymbol \beta$ is $(2,2,\dots, 2)$.  It is easy to
see that, for the remaining $\boldsymbol n$, the associated
configuration $\boldsymbol \mu= (\mu_{1},\dots, \mu_{\nu})$ is such
that the length $\nu$ is at most $ \kappa/2-1$ distinct elements, and
satisfy $\mu_1\geq 4$.

The equality \eqref{valueA} and some combinatorial considerations lead
to the following equality
\begin{align}
  \Sigma_{1,M}& =3\cdot 5\cdots (\kappa-1)\Bigl( \sum_{1\leq
    |n|<p/2}\tau_{\star} (n)^2\,
  W\Bigl(\frac{n}{Y}\Bigr)^2\Bigr)^\frac{\kappa}{2}\nonumber
  \\
  & +O\Biggl(\, \sum_{1\leq \nu\leq \frac{\kappa}{2}-1} \ \sum_{
    \substack{
      \mu_1\geq\cdots \geq \mu_{\nu}\geq 2\\
      2\mid \mu_i,\, \mu_1\geq 4\\
      \mu_1+\cdots \mu_\nu =\kappa }}\, \prod_{i=1}^\nu \sum_{1\leq
    |n|<p/2}d(|n|)^{\mu_i} \, \Bigl|
  W\Bigl(\frac{n}{Y}\Bigr)^{\mu_i}\Bigr|
  \Biggr)\nonumber\\
  &=m_{\kappa}\Bigl( \sum_{1\leq |n|<p/2}\tau_{\star} (n)^2\,
  W\Bigl(\frac{n}{Y}\Bigr)^2\Bigr)^\frac{\kappa}{2} +
  O(Y^{\kappa/2-1+\eps})\label{1M}
\end{align}
for any $\epsilon>0$, the error term arising easily from \eqref{weak}
(recall that $m_{\kappa}$ is given by~(\ref{eq-mk}) and is the
$\kappa$-th moment of a standard Gaussian).  We therefore see that the
proof of Theorem~\ref{funda} is completed by combining \eqref{367},
\eqref{Sigma2}, \eqref{632} and \eqref{1M} together with
Proposition~\ref{pr-mainterm}, applied with $a=b=1$.

\subsection{Further remarks}\label{ssec-comments}

We compare here the estimate of Theorem~\ref{funda} with other bounds
for the moments which can be derived straightforwardly from earlier
results. For simplificity, we restrict our attention to the case of
cusp forms.
\par
First, we note that it is fairly easy to deduce from
Proposition~\ref{stack1} and from Proposition~\ref{TTT} that
\begin{equation}\label{bound}
  E_f(X, c, a) \ll_{ f}  \frac{c}{X^{1/2}}\,d(c)^{5/2},
\end{equation}
for any $c\geq 1$, $X\geq c$ and any integer $a$. When $c\leq
X^{2/3}$, this statement is better than the bound 
$$
E_f(X,c,a)\ll X^{1/2+\eps}c^{-1/2}
$$ 
coming from Deligne's estimate for $\rho_f(n)$ (this is very similar
to the result first proved by Smith~\cite[(4)]{Sm} which has the same
range of uniformity; see also the remarks in~\cite[p. 276]{Bl} and the
work of Duke and Iwaniec~\cite[Th. 2]{duke-iwaniec}).  Combining these
two bounds in the definition \eqref{defM} of $\mathcal M$, we obtain
$$ 
{\mathcal M}_{f} (X,c;\kappa)\ll_\epsilon
X^{\epsilon} \Bigl(\frac{c^2}{X}\Bigr)^{\kappa/2}
\min\Bigl(1,
\frac{X^2}{c^{3}} \Bigr)^{\kappa/2}.
$$
\par
However, for $\kappa\geq 2$, we can also write
$$
{\mathcal M}_{f} (X,c;\kappa) \leq \Bigl(\max_{a\bmod c}\bigl\vert E_f
(X,c,a) \bigr\vert \Bigr)^{\kappa -2} \Bigl( \frac{1}c\sum_{a\bmod c}\bigl\vert
E_f (X,c, a)\bigr\vert^2 \Bigr),
$$
and then using the result~(\ref{19:56}) of Lau and Zhao, we deduce a
second inequality
\begin{equation}\label{13:48}
  {\mathcal M}_{f} (X,c;\kappa)\ll_\epsilon 
  X^{\epsilon} 
  \Bigl(\frac{c^2}{X}\Bigr)^{\kappa/2-1},
\end{equation}
which holds uniformly for $X^\frac{1}{2}\leq c \leq X$.  We then see
that our result in Theorem \ref{funda}, for $c=p$ a prime, improves
\eqref{13:48} for
\begin{equation}\label{rain}
X^\frac{1}{2}<p<X^\frac{2}{3}  \text{ and } \kappa \geq 3.
\end{equation}

We conclude by noting that Theorem~\ref{funda} can be extended without
much effort to cusp forms $f$ of arbitrary level and nebentypus, which
are not necessarily Hecke forms.  On the other hand, it does not seem
straightforward to extend the result to an arbitrary composite modulus
$c\geq 1$.

\subsection{Proof of Corollary~\ref{cor-gaussian}}
\label{sec-gaussian}

Corollary~\ref{cor-gaussian} is an easy consequence of the fact that
convergence to a Gaussian can be detected by convergence of the
moments to the Gaussian moments (see, e.g.,~\cite[Th. 8.48,
Prop. 8.49]{Br}).  For $p$ prime, let
$$
X=p^2/\Phi(p),\quad\quad \Phi(p)\ra +\infty,\quad\quad \Phi(p)\ll
p^{\eps}.
$$
\par
Denoting
$$
\nmom_{\star}(X,p;\kappa)= \frac 1p\sum_{a\in\Fpt}
\Bigl(\frac{E_{\star}(X,p,a)} {\sqrt{c_{\star,w}}}\Bigr)^{\kappa},
$$
we see from Theorem~\ref{funda} that for any $\eps>0$, we have
$$
\nmom_{\star}(X,p;\kappa)=m_{\kappa}+O\Bigl(\Phi(p)^{-1/2+\eps}+
p^{-\frac{1}{2}+\epsilon}\Phi(p)^{\kappa/2}\Bigr)\lra m_{\kappa}
$$
as $p\rightarrow +\infty$. Since this holds for any fixed integer
$\kappa\geq 1$, this finishes the proof.

\begin{remark}\label{rm-clt}
  (1) If $X=p^{2-\delta}$ for some fixed $\delta>0$, we can not prove
  the Central Limit Theorem, but nevertheless, we still deduce that
  the $\kappa$-moments converge to Gaussian moments  when
$$
1\leq \kappa\leq \Bigl\lfloor \frac{1}{\delta}\Bigr\rfloor.
$$
\par
(2) In this result, the Gaussian moments arise in
Proposition~\ref{390}, and in fact the combinatorics of the
computation is the same as in a standard case of the Central Limit
Theorem, namely the convergence in distribution to a standard Gaussian
of a sequence
$$
\mathrm Y_n=\frac{2\cos(\mathrm X_1)+\cdots +2\cos(\mathrm X_n)}{\sqrt{n}}
$$
where the $(\mathrm X_i)$ are independent random variables (defined on some
probability space) distributed on $[0,\pi]$ according to the Sato-Tate
measure.
\par
(3) It is natural to expect that an asymptotic formula
\begin{equation}\label{201}
  \mathcal{M}_{\star}(X,p;\kappa)\sim C_{\star}(\kappa),
\end{equation}
should be true uniformly for any even $\kappa$,  and 
$$
X^{\frac{1}{2}+\eps} \leq p\leq X^{1-\delta },
$$
for some fixed $\delta \ (0<\delta <1/2)$, which (with a corresponding
upper-bound for the odd moments) would extend
Corollary~\ref{cor-gaussian} to this range. This conjecture is true
for $\kappa=2$ (by \eqref{19:56}), and is in agreement with the square
root cancellation philosophy~(\ref{srp}).
\par
Another partial indication in favor of this conjecture is that a lower
bound of that size holds: considering $\star=f$ for simplicity, and
taking $\kappa\geq 2$ even, we have
$$
{\mathcal M}_{f}(X,p;2)\leq \bigl(\,{\mathcal
  M}_{f}(X,p;\kappa)\,\bigr)^\frac{2}{\kappa}\cdot \Bigl(
\frac 1p\sum_{1\leq a
  \leq p}1\Bigr)^{1-\frac{2}{\kappa}},
$$
and, by combining this with \eqref{19:56}, we obtain the lower bound
$$
{\mathcal M}_{f}(X,p;\kappa)\gg
1
$$
uniformly for $X^\frac{1}{2} \leq c\leq X^{1-\delta}$.
\end{remark}

\section{Proof of Theorem \ref{MIXEDMOMENTS} }\label{PROOFMIXED}

The proof of this Theorem has many similarities with the proof of
Theorem \ref{funda}, particularly in the computation of the error
terms. We will mainly concentrate on the study of the main term of the
mixed moment ${\mathcal M} _\star(X,p; \kappa , \lambda; \gamma)$.

We suppose that \eqref{X1/2<p<X} is satisfied and that $p$ is
sufficiently large in terms of $\gamma$. We start from the definition
\eqref{defMmix} and apply the same computations leading to
\eqref{358*}, \eqref{367} and \eqref{Sigma2} to write the equality
\begin{align}\label{fifi}
  &{\mathcal M} _\star(X,p; \kappa , \lambda;
  \gamma)=\frac{1}{pY^\frac{\kappa +\lambda}{2}} \bigl( \Sigma_3
  +O_\delta(X^{-1})\bigr)
\end{align}
where
\begin{align}
  \Sigma_3= &\underset{1\leq \vert m_1\vert,\dots,\vert m_\kappa\vert
    <p/2 }{\sum\cdots \sum} \tau_\star (m_1)\cdots \tau_\star
  (m_\kappa)
  W\Bigl( \frac{m_1}{Y}\Bigr)\cdots W\Bigl( \frac{m_\kappa}{Y}\Bigr)\label{123}\\
  &\times \underset{1\leq \vert n_1\vert,\dots,\vert n_\lambda\vert
    <p/2 }{\sum\cdots \sum} \tau_\star (n_1)\cdots \tau_\star
  (n_\lambda) W\Bigl( \frac{n_1}{Y}\Bigr)\cdots
  W\Bigl( \frac{n_\lambda}{Y}\Bigr)\nonumber \\
  &\times \sum_{\substack{1\leq a < p\\ a,\gamma\cdot
      a\not=0,\infty}}\hypk(m_{1 }a;p)\cdots \hypk(
  m_{\kappa}a;p)\hypk(n_1(\gamma \cdot a) ;p)\cdots
  \hypk(n_{\lambda}(\gamma\cdot a) ;p).\nonumber
\end{align}
\par
Since $p$ divides none of the $m_i$ or $n_j$, we see that the inner
sum over $a$ is equal to $ {\mathfrak S} (\kappa+\lambda, \boldsymbol
\beta, p)$, as defined in Proposition \ref{390}, where
\begin{equation}\label{defbeta1}
  \boldsymbol \beta =\bigl(h_{m_1} ,\dots, h_{m_\kappa}, 
  h_{n_1}\circ\gamma, \dots, h_{n_\lambda}\circ \gamma\bigr),
\end{equation}
and $h_m$ denotes the homothety
$$
h_m=\begin{pmatrix} m& 0\\0&1
\end{pmatrix}\in\PGL_2(\Fp).
 $$
\par
To apply Proposition \ref{390}, we have to understand which
$\boldsymbol \beta$ are in mirror configuration, in the sense of
Definition \ref{defconf}. This depends on whether $\gamma$ is diagonal
or not. 

\subsection{When $\gamma$ is not diagonal} 
\label{casenondiag}

If $\gamma$ is not a diagonal matrix, then 
$$
h_{m_{i}}\not=h_{n_{j}}\circ \gamma
$$ 
for any $i=1,\dots ,\kappa$ and for any $j=1,\dots ,\lambda$.  Hence,
in that case, the configuration of $\bfbeta$ defined by
\eqref{defbeta1} has (before ordering the elements by decreasing
order) the shape
$$
(\bfmu,\bfmu')=(\mu_1,\cdots,\mu_\nu,\mu'_1,\cdots,\mu'_{\nu'})
$$ 
where 
$$
\bfmu=(\mu_1,\cdots,\mu_\nu),\quad\quad
\bfmu'=(\mu'_1,\cdots,\mu'_{\nu'})
$$ 
are the configurations of 
$$
\bigl(h_{m_1} ,\dots, h_{m_\kappa}\bigr),\quad\quad
\bigl(h_{n_{1}}\circ\gamma, \dots, h_{n_{\lambda}}\circ\gamma\bigr),
$$
respectively.  It follows from Proposition \ref{390} that
\begin{align}
  {\mathfrak S} (\kappa+\lambda, \boldsymbol \beta, p)&=
  \sum_{\substack{1\leq a < p\\ a,\gamma\cdot
      a\not=0,\infty}}\hypk(m_{1 }a;p)\cdots \hypk(
  m_{\kappa}a;p)\hypk(n_{1}(\gamma\cdot a);p)\cdots
  \hypk(n_{\lambda}(\gamma\cdot a)
  ;p)\nonumber\\
  &=A(\bfmu)A(\bfmu')\,p+O_{\kappa,\lambda}(p^{\frac12}).
\label{frf}
\end{align}
\par
Hence by \eqref{fifi}, \eqref{123} and \eqref{frf} and by computations
similar to those we did in \S \ref{4.1}, we deduce the equality
\begin{equation}\label{MMM} {\mathcal M} _\star(X,p; \kappa , \lambda;
  \gamma)=
  Y^{-\frac{\kappa+\lambda}{2}}\Bigl(\Sigma_{3,M}(\kappa)
  \Sigma_{3,M}(\lambda)+O(p^{-\frac{1}{2}+\eps}
  Y^{\kappa+\lambda})\Bigr)+O(p^{-1}),
\end{equation}
with
\begin{align*}
  \Sigma_{3,M}(\kappa)&=\underset{1\leq \vert m_1\vert,\dots,\vert
    m_\kappa\vert <p/2 }{\sum\cdots \sum} \tau_\star (m_1)\cdots
  \tau_\star (m_\kappa) W\Bigl( \frac{m_1}{Y}\Bigr)\cdots W\Bigl(
  \frac{m_\kappa}{Y}\Bigr)A(\bfmu),\\
  \Sigma_{3,M}(\lambda)&=\underset{1\leq \vert n_1\vert,\dots,\vert
    n_\lambda\vert <p/2 }{\sum\cdots \sum} \tau_\star (n_1)\cdots
  \tau_\star (n_\lambda) W\Bigl( \frac{n_1}{Y}\Bigr)\cdots W\Bigl(
  \frac{n_\lambda}{Y}\Bigr)A(\bfmu').
\end{align*}
\par
If $\kappa$ or $\lambda$ is odd, the product $A(\bfmu)A(\bfmu')$ is
zero, hence~\eqref{nondiag} follows in that case. If $\kappa$ and
$\lambda$ are both even, then as in \eqref{1M}, we prove that the
largest contribution comes from the case where $\boldsymbol \mu
=(2,\dots, 2)$ and $\boldsymbol \mu'= (2,\dots,2)$. Hence, by a
computation similar to \eqref{1M} and \eqref{mark}, we get the
equality
$$
\Sigma_{3,M} (\kappa) =\Bigl\{C_{\star} (\kappa) +
O\bigl(Y^{-\frac{1}{2} +\epsilon}\bigr) \Bigr\} Y^\frac{\kappa}{2},
$$
and a similar one for $\Sigma_{3,M} (\lambda)$. Hence, by \eqref{MMM},
we complete the proof of \eqref{nondiag}.

\subsection{When $\gamma$ is diagonal } \label{casediag}

We then write $\gamma$ in the canonical form \eqref{canonical} and we
suppose that 
$$
p>\max(\vert \gamma_{1} \vert, \,\vert
\gamma_{2}\vert).
$$
\par
Then, by making the change of variable $a=\gamma_{2}a'$, we find that
the sum over $a$ of normalized Kloosterman sums appearing in the last
line of \eqref{123} is equal to $ {\mathfrak S} (\kappa+\lambda,
\boldsymbol \beta, p)$ as defined in Proposition \ref{390}, with
\begin{equation}\label{defbeta}
  \boldsymbol \beta =\bigl(h_{\gamma_{2}m_{1}} ,\dots, 
  h_{\gamma_{2} m_\kappa}, h_{\gamma_1 n_1}, \dots, h_{\gamma_1 n_\lambda}\bigr).
\end{equation}
\par
If the configuration of $\bfbeta$ is not a mirror configuration, we have
$$
{\mathfrak S} (\kappa+\lambda, \boldsymbol \beta, p)=O(p^{\frac12}).
$$
\par
In particular, if $\kappa \not \equiv \lambda \bmod 2$, we deduce by
\eqref{123}, \eqref{weak} and by similar treatment of the error terms
as above, that
\begin{equation}\label{toxi}
\Sigma_{3}\ll p^{\frac12 +\eps} Y^{\kappa+\lambda}.
\end{equation}
\par
Combining this with \eqref{fifi} we complete the proof of
\eqref{diag} when $\kappa$ and $\lambda$ have opposite parity.
\par
Now assume that $\kappa$ and $\lambda$ have same parity. The
combinatorics involved is then more delicate than in \S
\ref{casenondiag}, because me must take into account the cases of {\it
  crossed mirror configurations}, namely situations when some of the
$\gamma_{2}m_{i}$ are equal to some of the $\gamma_{1}n_{j}$.  
\par
To be precise, we can decompose $\Sigma_{3}$ (see \eqref{123}) into
\begin{equation}\label{decompSigma3}
  \Sigma_{3}=B^{\rm nm} +
  B_{0}^{\rm m}+
  \sum_{\substack{0\leq \nu\leq
      \min(\kappa,\lambda)\\\nu\equiv\kappa\equiv\lambda\bmod 2}}
  B^{\rm m}(\nu),
\end{equation}
where
\begin{itemize}
\item $B^{\rm nm}$ corresponds to the contribution of the
  $(\gamma_{2}m_{1},\dots, \gamma_{2}m_{\kappa},
  \gamma_{1}n_{1},\dots, \gamma_{1}n_{\lambda})$ which are not in
  mirror configuration,
\item $B_{0}^{\rm m}$ corresponds to the contribution of the
  $(\gamma_{2}m_{1},\dots, \gamma_{2}m_{\kappa},
  \gamma_{1}n_{1},\dots, \gamma_{1}n_{\lambda})$ which are in mirror
  configuration, but that configuration is not $(2,\dots, 2)$,
\item $B^{\rm m}(\nu)$ corresponds to the contribution of the
  $(\gamma_{2}m_{1},\dots, \gamma_{2}m_{\kappa},
  \gamma_{1}n_{1},\dots, \gamma_{1}n_{\lambda})$ which have a mirror
  configuration equal to $(2,\dots, 2)$, and where exactly $\nu$ of
  the $\gamma_{2}m_{i}$ ($1\leq i \leq \kappa$) are equal to $\nu$ of
  the $n_{\nu}n_{j}$ ($1\leq j \leq \lambda$).
\end{itemize}
\par
The same computation as for \eqref{toxi} gives the relation
\begin{equation*}
  B^{\rm nm}\ll p^{\frac{1}{2} +\eps}Y^{\kappa +\lambda},
\end{equation*}
which, when combined with \eqref{fifi}, fits with the error term in
\eqref{eq-mixed-conv}.
\par
We can also estimate $B_0^{\rm m} $ by following the same technique
which led to the error term in \eqref{1M}, and obtain
$$
B_{0}^{\rm m} \ll p Y^{\frac{\kappa +\lambda}{2} -1 +\eps},
$$
which, by \eqref{fifi} is absorbed by the error term in \eqref{eq-mixed-conv}.
\par
The case of $B^{\rm m}(\nu)$ is more delicate to treat. For the terms
in that sum, exactly $\nu$ of the $\gamma_{2}m_{i}$ ($1\leq i \leq
\kappa$) are equal to $\nu$ of the $\gamma_{1}n_{j}$ ($1\leq j \leq
\lambda$), and the remaining $\gamma_{2}m_{i}$ (resp. $\gamma_{1}n_j$)
are in configuration $(2,\dots, 2)$. 
The condition $\gamma_{2}m_{i} =\gamma_{1}n_{j}$ can be parametrized
by $m_{i}=\gamma_{1}t $ and $n_j=\gamma_{2}t$ where $t$ is a non-zero
integer.  Appealing to Proposition \ref{390}, and applying some
combinatorial considerations, we deduce the formula
\begin{multline}\label{1654}
  B^{\rm m}(\nu)=p\, \nu ! \binom{\kappa}{\nu}
  \binom{\lambda}{\nu}\Bigl( \sum_{1\leq \vert \gamma_{1} t\vert,\,
    \vert \gamma_{2}t \vert < p/2} \tau_{\star} (\gamma_{1}t)
  \tau_{\star}(\gamma_{2}t)W\Bigl( \frac{\gamma_{1} t}{Y}\Bigr) W\Bigl(
  \frac{\gamma_2 t}{Y}\Bigr)
  \Bigr)^\nu\\
  \times \bigl(1\cdot 3\cdots (\kappa-\nu -1)\bigr) \Bigl( \sum_{1\leq
    \vert m\vert < p/2}\tau_{\star}^2
  (m) W^2\Bigl( \frac{m}{Y}\Bigr) \Bigr)^\frac{\kappa-\nu}{2}\\
  \times \bigl(1\cdot 3\cdots (\lambda-\nu -1)\bigr)\Bigl( \sum_{1\leq
    \vert n\vert < p/2} \tau_{\star}^2 (n) W^2\Bigl( \frac{n}{Y}\Bigr)
  \Bigr)^\frac{\lambda- \nu}{2}+O\bigl( p^{\frac{1}{2} +\eps}Y^\frac{\kappa +\lambda}{2}\bigr).
\end{multline}
\par
In this expression, the first term corresponds to the choice and to
the contribution of the $\nu$ integers $m_{i} $ and $\nu$ integers
$n_{j}$ which satisfy the condition
$\gamma_{2}m_{i}=\gamma_{1}n_{j}$. The second factor corresponds to
the contribution of the $\kappa -\nu$ remaining $m_{i}$ which are in
configuration $(2,\dots,2)$ between themselves, and the third factor
to the $\lambda-\nu$ remaining $n_{j}$ in configuration $(2,\dots,2)$
between themselves. Finally, the error term comes from the error term
in \eqref{579}. 
\par
Using the arithmetic sums $\mathcal B_{\star} (m,n,Y)$ defined in
Proposition~\ref{pr-mainterm}, we can thus summarize \eqref{1654} in
the form
\begin{equation}\label{1757}
  B^{\rm m}(\nu)=p\, \frac{\kappa!\,\lambda!}{\nu!\,
    2^{\frac{\kappa+\lambda}{2}
      -\nu} \,((\kappa-\nu)/2)! \,((\lambda-\nu)/2)!} \\
  \mathcal B_{\star} (1,1,Y)^{\frac{\kappa
      +\lambda}{2}-\nu}\, \mathcal
  B_{\star}(\gamma_{1},\gamma_{2},Y)^{\nu} + O\bigl(
  p^{\frac{1}{2} +\eps}Y^\frac{\kappa +\lambda}{2}\bigr).
\end{equation}
\par
We now obtain \eqref{diag} by combining \eqref{defY}, \eqref{fifi},
\eqref{decompSigma3}, \eqref{1757} and Proposition~\ref{pr-mainterm}.

\section{Proof of Corollary \ref{gaussianvector} }

We now deduce Corollary~\ref{gaussianvector} from
Theorem~\ref{MIXEDMOMENTS}. The probabilistic tool is the following
standard lemma:

\begin{lemma}
  Let $(\rand{X}_n,\rand{Y}_n)$ be a sequence of real-valued random
  variables. Let $Q$ be a positive definite symmetric $2\times 2$
  matrix. Suppose that, for any integers $\lambda$, $\kappa\geq 0$, we
  have
$$
\expect(\rand{X}_n^{\kappa}\rand{Y}_n^{\lambda})\longrightarrow
m_{\kappa,\lambda}(Q)
$$
as $n\ra +\infty$, where
$m_{\kappa,\lambda}(Q)=\expect(\rand{A}^{\kappa}\rand{B}^{\lambda})$
for some centered gaussian vector $(\rand{A},\rand{B})$ with
covariance matrix $Q$. Then $(\rand{X}_n,\rand{Y}_n)$ converges in law
to $(\rand{A},\rand{B})$.
\end{lemma}

This follows from the case of individual sequences using the
characterization of the Gaussian vector $(\rand{A},\rand{B})$ by its
linear combinations $\alpha \rand{A}+\beta \rand{B}$ being Gaussian.
\par
We apply this lemma to the sequence $(\rZ_p,\rZ_p\circ \gamma)$ for
$p$ prime, as in the statement of Corollary~\ref{gaussianvector}. Note
that if $\star=d$, the main term $C_{d}(\kappa,\lambda,\gamma)$ still
depends on $p$ (because of the polynomials of $(\log p^2/X)$ which it
involves). However, under the assumptions of
Corollary~\ref{gaussianvector} on $X$ and $p$, we see that in all
cases, for fixed $\kappa\geq 0$ and $\lambda\geq 0$, the limit
$$
L_{\kappa,\lambda}=\lim_{p\ra+\infty}
\frac{C_{\star}(\kappa,\lambda,\gamma)}{(c_{\star,w})^{(\kappa+\lambda)/2}}
$$
exists, and that
\begin{equation}\label{eq-limits}
  \lim_{p\ra +\infty} \expect(\rZ_p^{\kappa}(\rZ_p\circ\gamma)^{\lambda})=
  L_{\kappa,\lambda}.
\end{equation}
\par
If $\gamma$ is not diagonal, we get by~(\ref{nondiag})
and~(\ref{eq-cfk}) that $L_{\kappa,\lambda}=m_{\kappa}m_{\lambda}$
which coincides obviously with the mixed moment
$\expect(\rand{A}^{\kappa}\rand{B}^{\lambda})$ where
$(\rand{A},\rand{B})$ are independent centered Gaussian variables with
variance $1$, so we obtain Corollary~\ref{gaussianvector} in that
case.
\par
If $\gamma$ is diagonal, we must check that $L_{\kappa,\lambda}$
corresponds to the mixed moments of a gaussian vector
$(\rand{A},\rand{B})$ with covariance matrix given
by~(\ref{eq-cor-diag}). For this purpose, we use the
formula~(\ref{diag}) and note that
$$
\frac{(c_{\star,w})^{\frac{\kappa+\lambda}{2}-\nu}\, (\tilde c_{\star,
    w,\gamma})^\nu}{c_{\star,w}^{(\kappa+\lambda)/2}}
=\Bigl(\frac{\tilde{c}_{\star,w,\gamma}}{c_{\star,w}}\Bigr)^{\nu}\ra
(\cov_{\star,\gamma,w})^{\nu}
$$
as $p\ra +\infty$, with notation as in~(\ref{diag}) and
Corollary~\ref{gaussianvector}. Thus, abbreviating
$\cov=\cov_{\star,\gamma,w}$, we compute the $2$-variable exponential
generating series of $L_{\kappa,\lambda}$ by writing
\begin{align*}
  \sum_{\kappa,\lambda\geq 0}{\frac{1}{\kappa!\lambda!}
    L_{\kappa,\lambda}U^{\kappa}V^{\lambda}}&=
  \sum_{\kappa,\lambda\geq 0}
  \frac{1}{\kappa!\lambda!}U^{\kappa}V^{\lambda} \sum_{\substack{0\leq
      \nu\leq \min (\kappa,\lambda)\\ \nu\equiv \kappa\equiv \lambda
      \bmod 2}} \nu ! \binom{\kappa}{\nu} \,\binom{\lambda}{\nu}\,
  m_{\kappa-\nu}\,m_{\lambda-\nu} \cov^{\nu}\\
  &=\sum_{\nu\geq 0}\nu!\cov^{\nu} \sum_{k,l\geq
    0}{\frac{U^{\nu+2k}V^{\nu+2l}}{(\nu+2k)!(\nu+2l)!}
    \binom{\nu+2k}{\nu}\binom{\nu+2l}{\nu}m_{2k}m_{2l}}\\
  &=\sum_{\nu\geq 0}\frac{\cov^{\nu}(UV)^{\nu}}{\nu!}  \sum_{k\geq
    0}\frac{m_{2k}U^{2k}}{(2k)!}  \sum_{l\geq
    0}\frac{m_{2l}V^{2l}}{(2l)!}
  =\exp\Bigl(\frac{U^2}{2}+\cov UV+\frac{V^2}{2}\Bigr).
\end{align*}
\par
Since this is well-known to be the exponential generating series of
the moments of the Gaussian vector with covariance
matrix~(\ref{eq-cor-diag}), we obtain the desired convergence in law.

\end{document}